\documentclass[11pt,letterpaper]{article}

\usepackage{times,amssymb,amsthm,amsmath,multirow,amsfonts,enumerate,bbm,setspace,pdflscape,lscape,mathtools,comment}
\RequirePackage{graphicx,paralist,microtype,dsfont,longtable,booktabs,subcaption,xcolor,ragged2e,float}
\usepackage[T1]{fontenc}
\usepackage[margin=1in]{geometry}
\usepackage[compact]{titlesec}
\usepackage[normalem]{ulem}
\usepackage[round,authoryear]{natbib}
\usepackage[colorlinks=true,linkcolor=red,urlcolor=blue,citecolor=blue,anchorcolor=blue,
pdftex,breaklinks,pdfencoding=auto,psdextra,
bookmarksopenlevel=1,bookmarksopen=true]{hyperref}

\newcommand*{\SC}{\mathcal S_{\text{\scalebox{1.2}{$1$}}}}
\newcommand*{\SCC}{\mathcal S_{\text{\scalebox{1.2}{$2$}}}}
\newcommand*{\SCCC}{\mathcal S_{\text{\scalebox{1.2}{$p$}}}}

\newtheorem{theorem}{Theorem}[section]
\newtheorem{lemma}{Lemma}[section]
\newtheorem{corollary}{Corollary}[section]

\newtheorem{proposition}{Proposition}[section]
\newtheorem{remark}{Remark}[section]
\newtheorem{assumption}{Assumption}
\newtheorem{definition}{Definition}
\newtheorem{example}{Example}

\newtheorem*{assumption*}{Assumption}
\newtheorem*{proof*}{Proof of}

\DeclareMathOperator{\ran}{ran}

\DeclareMathOperator{\spn}{span}

\numberwithin{equation}{section}

\makeatletter
\renewcommand*{\eqref}[1]{\hyperref[{#1}]{\textup{\tagform@{\ref*{#1}}}}}

\newcommand{\AO}{A_{\mathrm{o}}}
\makeatother

\expandafter
\def \expandafter \normalsize \expandafter{\normalsize \setlength \abovedisplayskip{10pt plus 2pt minus 6pt}}
\expandafter
\def \expandafter \normalsize \expandafter{\normalsize \setlength \abovedisplayshortskip{0pt plus 2pt}}
\expandafter
\def \expandafter \normalsize \expandafter{\normalsize \setlength \belowdisplayskip{10pt plus 2pt minus 6pt}}
\expandafter
\def \expandafter \normalsize \expandafter{\normalsize \setlength \belowdisplayshortskip{5pt plus 2pt minus 3pt}}

\newcommand{\commRV}{\textcolor{black}}
\renewcommand\footnotemark{}
\newcommand{\revlab}[1]{\phantomsection\label{#1}} 
\begin{document}
\date{}
\title{\vspace{-0em}\Large Optimal linear prediction with functional observations: Why you can use a simple post-dimension reduction estimator}
\author{	\large	Won-Ki Seo$^\ast$\thanks{$^\ast$This paper has benefited greatly from the generosity and insight of the Co-editor and two anonymous referees. The R code used in this paper is available at the author's website. E-mail address: {won-ki.seo@sydney.edu.au}} 
	\\ University of Sydney	} 
\maketitle \vspace{-0em} 
\begin{abstract}
	\vspace{-.3em}		We study the optimal linear prediction of a random function that takes values in an infinite dimensional Hilbert space. We begin by characterizing the mean square prediction error (MSPE) associated with a linear predictor and discussing the minimal achievable MSPE. This analysis reveals that, in general, there are multiple non-unique linear predictors that minimize the MSPE, and even if a unique solution exists, consistently estimating it from finite samples is generally impossible. Nevertheless, we can define asymptotically optimal linear operators whose empirical MSPEs approach the minimal achievable level as the sample size increases. We show that, interestingly, standard post-dimension reduction estimators, which have been widely used in the literature, attain such asymptotic optimality under minimal conditions.  \\[2pt]
	\noindent 	 \textbf{MSC 2020}: 60G25. \\ 
	\medskip	 \noindent \textbf{Keywords}: Linear prediction; functional data; functional linear models; regularization.		 \vspace{0em}	
\end{abstract}

\newpage 


\section{Introduction} \label{secintro} 
We study the optimal linear prediction in an arbitrary Hilbert space \(\mathcal H\), and how to estimate the optimal linear predictor. Given recent developments in functional data analysis, studies on this subject hold significant importance and relevance to many empirical applications; see e.g.,    \cite{Park2012397}, \cite{aue2015prediction},  \cite{aue2017estimating}, \cite{klepsch2017prediction}  and   \cite{KLEPSCH2017252} to name only a few recent papers. The reader is also referred to \cite{BOSQ2007879}, \cite{bosq2014computing} and \cite{mollenhauer2023learning} containing an earlier mathematical exploration of this subject. 

Let \(\{Y_t\}_{t\geq 1}\) and \(\{X_t\}_{t\geq 1}\) be stationary sequences of mean-zero random elements, both taking values in \(\mathcal H\). Suppose that \(\widetilde{Y}_t = \AO X_t\), where \(\AO\) is a continuous linear operator, and it satisfies the following: 
for any arbitrary continuous linear operator \(B\),
\begin{equation}
	\mathbb{E}\|Y_t - \widetilde{Y}_t\|^2 = \mathbb{E}\|Y_t - \AO X_t\|^2 \leq \mathbb{E}\|Y_t - BX_t\|^2,\label{eqmse}
\end{equation}
where \(\|\cdot\|\) is the norm defined on \(\mathcal H\). We refer to \(\widetilde{Y}_t = \AO X_t\) as an optimal linear predictor (OLP) and  \(\AO\) as an optimal linear prediction operator (OLPO). Given the observations \(\{Y_t,X_t\}_{t=1}^T\), practitioners are often interested in constructing a predictor \(\widehat{Y}_t\) that converges (in probability) to an OLP as \(T\) increases. 
This is straightforward when \(\mathcal H = \mathbb{R}\), and \(Y_t\) and \(X_t\) are real-valued mean-zero random variables with positive variances. In this case, it is well known that the unique solution to \eqref{eqmse} is achieved by \(\AO  = \mathbb{E}[Y_tX_t]/\mathbb{E}[X_t^2]\) and the minimal achievable mean squared prediction error (MSPE) is \(\mathbb{E}\|Y_t - \AO X_t\|^2 = \mathbb{E}[Y_t^2]-(\mathbb{E}[Y_tX_t])^2/\mathbb{E}[X_t^2]\). The conventional plug-in-type estimator (or OLS estimator) of \(\AO \) may be defined by 
$\widehat{A}={T^{-1}\sum_{t=1}^T Y_t X_t}/{ T^{-1}\sum_{t=1}^T  X_t^2}.$
When the weak law of large numbers holds for $\{Y_t^2\}_{t\geq 1}$, $\{X_t^2\}_{t\geq 1}$ and $\{X_tY_t\}_{t\geq 1}$, we find that the following two results hold: (i) 	$\widehat{A}   \to_p    \AO = \mathbb{E}[X_tY_t]/\mathbb{E}[X_t^2]$ and (ii) $T^{-1} \sum_{t=1}^T  (Y_t - \widehat{A}X_t)^2  \to_p    \mathbb{E}[Y_t^2] - (\mathbb{E}[X_tY_t])^2/\mathbb{E}[X_t^2]$, 
where and hereafter \(\to_p\) denotes the convergence in probability with respect to the norm of \(\mathcal H\). That is, a consistent estimator of the OLPO can be constructed from the given observations. Furthermore, the empirical MSPE obtained from the estimator converges to the minimal achievable MSPE. 

However, in a more general situation where \(Y_t\) and \(X_t\) take values in a possibly infinite dimensional \(\mathcal H\), it is generally impossible to obtain parallel results. To see this with a simple example, suppose that \(\{Y_t,X_t\}_{t\geq 1}\) satisfies the following: for \(t \geq 1\), 
\begin{align}
	Y_t = \AO X_t + \varepsilon_t, \quad \AO = \sum_{j \geq 1} a_j f_j \otimes f_j, \quad \underline{a} \leq |a_j| \leq \overline{a}, \quad  \underline{a},  \overline{a}>0,  \label{eq003}
\end{align}
where $\varepsilon_t$ is independent of \(X_t\) and $\{f_j\}_{j\geq 1}$ is an orthonormal basis of $\mathcal H$. \(\AO\) specified in \eqref{eq003} is obviously the OLPO, but \(\AO\) is not consistently estimable in general. As will be detailed in Example \ref{ex1}, this is because, it is not possible to estimate \(a_{j}\) for \(j > T\) from \(T\) observations unless (i) a simplifying condition on \(A\), such as \(a_j=a\) for all \(j \geq m\) for some finite \(m\), holds and (ii) researchers are aware of this condition and use it appropriately. Even in the simple case where \(Y_t = \bar{\AO} X_t + \varepsilon_t\) with  \(\bar{\AO} = a I\)  (\(I\) denotes the identity map)  for \(a \in \mathbb{R}\), consistent estimation of \(\bar{\AO}\) is impossible for the same reason if we do not know such a simple structure of \(\bar{\AO}\) and hence allow a more general case given in \eqref{eq003} (note that, \(\bar{\AO} = a I\) is a special case of \(\AO\) in  \eqref{eq003} with \(a_j=a\)). 

Does this mean that it is impossible to statistically solve the optimal linear prediction problem in this general setting? The answer is no. We will demonstrate that, under mild conditions, there exists the minimum MSPE achievable by a linear predictor and it is feasible to construct a possibly inconsistent estimator \(\widehat{A}\) such that the empirical MSPE, computed as \(T^{-1} \sum_{t=1}^T \|Y_t-\widehat{A}X_t\|^2\), converges to the minimum MSPE.
Particularly, we show that a standard post dimension-reduction estimator, obtained by (i) reducing the dimensionality of the predictive variable \(X_t\) using the principal directions of its sample covariance and then (ii) applying the least squares method to estimate the linear relationship between the resulting lower dimensional predictive variable and \(Y_t\), is effective for linear prediction. This post dimension-reduction estimator has been widely used due to its simplicity. Its statistical properties have been studied under technical assumptions, which are challenging to verify, such as those concerning the eigenstructure of the covariance of \(X_t\); the reader is referred to, e.g., \cite{imaizumi2018} and \cite{seong2021functional}, where the assumptions of \cite{Hall2007} are adopted for function-on-function regression models. 
We show that, without such assumptions, a naive use of this simple post dimension-reduction estimator can be justified as a way to obtain a solution which asymptotically minimizes the MSPE. We also extend this finding to show that similar estimators, which involve further dimension reduction of \(Y_t\), also possess this desirable property under mild conditions.

The paper proceeds as follows: Section \ref{sec_blp} characterizes the minimal achievable MSPE by a linear predictor, followed by a discussion on the estimation of an asymptotically optimal predictor in Section \ref{secest1}. Further discussions and extensions are  given in Section \ref{sec_disext}. 
Section \ref{secest2} provides simulation evidence of our theoretical findings, and Section \ref{secconclusion} contains concluding remarks. The Appendix includes mathematical proofs of the theoretical results and some additional simulation results.

\section{Optimal linear prediction in Hilbert space}\label{sec_blp}
\subsection{Preliminaries}
For the subsequent discussion, we introduce notation. Let $\mathcal H$ be a separable Hilbert space with inner product $\langle \cdot, \cdot \rangle$ and norm $\|\cdot\|$. For $V\subset\mathcal H$,  let $V^\perp$ be the orthogonal complement to $V$. We let $\mathcal{L}_{\infty}$ be the set of continuous linear operators, and let $\|\mathcal T\|_{\infty}$, $\mathcal T^\ast$, $\ran \mathcal T$, and $\ker\mathcal T$ denote the operator norm, adjoint, range, and kernel of $\mathcal T$, respectively. 
$\mathcal T$ is self-adjoint if $\mathcal T=\mathcal T^\ast$. $\mathcal T$ is called nonnegative  if $\langle \mathcal Tx,x\rangle \geq 0$ for any $x \in \mathcal H$, and positive if also $\langle \mathcal Tx,x\rangle \neq 0$ for any $x \in \mathcal H\setminus\{0\}$. For $x,y \in \mathcal H$, we let $x\otimes y$ be the operator given by $z \mapsto \langle x,z \rangle y$ for $z\in \mathcal H$. $\mathcal T \in \mathcal L_{\infty}$ is compact if $\mathcal T = \sum_{j\geq 1} a_jv_{j} \otimes w_{j}$ for some orthonormal bases $\{v_{j}\}_{j \geq 1}$ and $\{w_{j}\}_{j \geq 1}$ and a sequence of nonnegative numbers $\{a_j\}_{j \geq 1}$ tending to zero; if $\mathcal T$ is also self-adjoint and nonnegative (see  \citealp{Bosq2000}, p.\ 35), we may assume that $v_{j} = w_{j}$. For any compact $\mathcal T\in \mathcal L_{\infty}$ and $p \in \mathbb{N}$, let $\|\mathcal T\|_{\mathcal S_p}$ be defined by $\|\mathcal T\|_{\mathcal S_p}^p  =  {\sum_{j=1}^\infty a_j^p}$ and let \(\SCCC\) be the set of compact operators \(\mathcal T\) with \(\|\mathcal T\|_{\mathcal S_p} < \infty\);  \(\SCCC\) is called the Schatten \(p\)-class. \(\SC\) (resp.\ \(\SCC\)) is also referred to the trace (resp.\ Hilbert-Schmidt) class. It is known that the following hold: \(\|\mathcal T\|_{\infty} \leq \|\mathcal T\|_{\SCC} \leq \|\mathcal T\|_{\SC}\) and \(\|\mathcal T\|_{\SCC}^2 = \sum_{j=1}^\infty \|\mathcal T w_j \|^2\) for any orthonormal basis \(\{w_j\}_{j\geq 1}\). For any $\mathcal H$-valued mean-zero random elements $Z$ and $\widetilde{Z}$ with $\mathbb{E}\|Z\|^2 < \infty$ and $\mathbb{E}\|\widetilde{Z}\|^2 < \infty$, their cross-covariance $C_{Z\tilde{Z}}=\mathbb{E}[Z \otimes \tilde{Z}]$ is a Schatten 1-class operator; if $Z=\tilde{Z}$, it reduces to the covariance of $Z$ and $\mathbb{E}\|Z\|^2 = \|C_{ZZ}\|_{\SC}$ holds. 

\subsection{Linear prediction in $\mathcal H$ }	
Consider a weakly stationary sequence  \(\{Y_t,X_t\}_{t\geq 1}\) with nonzero covariances \(C_{YY}=\mathbb{E}[Y_t\otimes Y_t]\) and \(C_{XX}=\mathbb{E}[X_t\otimes X_t]\), along with the cross-covariance \(C_{XY}=\mathbb{E}[X_t\otimes Y_t]\) (or equivalently \(C_{YX}^\ast\)).   
We hereafter write \({C}_{YY}\) and \({C}_{XX}\) as their spectral representations as follows, if necessary: 
\allowdisplaybreaks{
	\begin{align}
		{C}_{YY} = \sum_{j\geq 1} \kappa_j u_j \otimes u_j, \qquad 	{C}_{XX} =  \sum_{j\geq 1} \lambda_j v_j \otimes v_j, \label{specXX}
\end{align}}where  $\kappa_1\geq \kappa_2\geq \ldots \geq 0$, $\lambda_1\geq \lambda_2\geq \ldots \geq 0$, and $\{u_j\}_{j\geq 1}$ and $\{v_j\}_{j\geq 1}$ are orthonormal sets of $\mathcal H$. Unless otherwise stated, we assume that \(C_{YY}\) and \(C_{XX}\) are not finite rank operators and thus there are infinitely many nonzero eigenvalues in \eqref{specXX}, which is as usually assumed for covariances of Hilbert-valued random elements in the literature on functional data analysis. 

Note first that, for any \(B \in \mathcal L_{\infty}\), $\mathbb{E}\|Y_t-BX_t\|^2 = \|\mathbb{E}[(Y_t-BX_t)\otimes (Y_t-BX_t)]\|_{\SC}$ and hence the MSPE associated with \(B\) can be written as follows:
\begin{equation}
	\mathbb{E}\|Y_t-BX_t\|^2  = \|C_{YY} - C_{XY}B^\ast - B C_{XY}^\ast   +  BC_{XX}B^\ast\|_{\SC}.  \label{eqmspe}
\end{equation} 
Next, we provide the main result of this section, which not only gives us a useful characterization of the MSPE in \eqref{eqmspe}, but also provides essential preliminary results for the subsequent discussion. 
\begin{proposition}\label{prop1} 
	For any $B \in \mathcal L_{\infty}$, there exists a unique element $R_{XY}\in \mathcal L_{\infty}$ such that $C_{XY}  = {C}_{YY}^{1/2}R_{XY} {C}_{XX}^{1/2}$ and 
	\begin{equation}	\label{eqprop1}
		\mathbb{E}\|Y_t-BX_t\|^2 =   \|C_{YY}- C_{YY}^{1/2}R_{XY}R_{XY}^\ast C_{YY}^{1/2}\|_{\SC} + \| BC_{XX}^{1/2} -  C_{YY}^{1/2}R_{XY}\|_{\SCC}^2.
	\end{equation}
\end{proposition}
Proposition \ref{prop1} shows that the MSPE associated with \(B \in \mathcal L_{\infty}\) is the sum of the Schatten 1- and 2-norms of specific operators dependent on $C_{YY}$, $C_{XX}$, $R_{XY}$ and $B$; notably, only the latter term ($\| BC_{XX}^{1/2} -  C_{YY}^{1/2}R_{XY}\|_{\SCC}^2$) in \eqref{eqprop1} depends on \(B\). Thus, the former term ($\|C_{YY}- C_{YY}^{1/2}R_{XY}R_{XY}^\ast C_{YY}^{1/2}\|_{\SC}$) represents the minimal achievable MSPE by a linear predictor, while the latter can be understood as a measure of the inadequacy of \(B\) as a linear predictor. If \(\{w_j\}_{j\geq 1}\) is an orthonormal basis of \(\mathcal H\), this inadequacy becomes zero if and only if 
\begin{equation}\label{eqprop1add}
	\|(B C_{XX}^{1/2} - C_{YY}^{1/2} R_{XY})w_j\|^2 = 0 \quad \text{for all $j \geq 1$.}
\end{equation}
Based on these findings, we obtain the following two characterizations of an OLPO in Corollary \ref{cor0}: the first is a direct consequence of \eqref{eqprop1add} and the Hahn-Banach extension theorem (see, e.g., Theorem 1.9.1 of \citealp{megginson1998}), which, in turn, implies the second due to the fact that  \(C_{XY}  = {C}_{YY}^{1/2}R_{XY} {C}_{XX}^{1/2}\) as observed in Proposition \ref{prop1}.
\begin{corollary}\label{cor0} $A$ is an OLPO if and only if any of the following equivalent conditions holds: (a)~$AC_{XX}^{1/2} = C_{YY}^{1/2}R_{XY}$ and (b)~$AC_{XX} = C_{XY}$.
\end{corollary}
Condition (\(b\)) follows directly from condition (\(a\)) and Proposition \ref{prop1}, and it is notably align with the characterization of an OLPO provided by \cite{BOSQ2007879}; Remark \ref{rem1} outlines the distinctions between our findings and the existing result in more detail. From Corollary \ref{cor0}, we find that the minimum MSPE is attained by \(A\in \mathcal L_{\infty}\) satisfying \(AC_{XX}=C_{XY}\) (or \(AC_{XX}^{1/2} = C_{YY}^{1/2}R_{XY}\)). However, such an operator \(A\) is not uniquely determined; particularly, the equation does not specify how \(A\) acts on \(\ker C_{XX}\), allowing \(A\) to agree with any arbitrary element in $\mathcal L_{\infty}$ on \(\ker C_{XX}\) (see Remark \ref{rem2}). When \(C_{XX}\) is not injective, there are multiple choices of \(A\) that achieve the minimum MSPE \commRV{(see Remark \ref{remadd} provide a more detailed discussion on the case where $C_{XX}$ is not injective.)}. In infinite dimensional settings, the injectivity of \(C_{XX}\)  is  a stringent assumption, and verifying this condition from finite observations is impractical.  Thus, pursuing prediction under the existence of the unique OLPO, as in the standard univariate or multivariate prediction, is restrictive. Even if a different setup is considered with a different purpose, similar concerns about the requirements for unique identification were recently raised by \cite{Babii2025}, and the enthusiastic reader is referred to their paper for more detailed discussion on the topic.

\begin{remark}\label{remadd} \normalfont
\commRV{If $C_{XX}$ is not injective, then $\ker C_{XX}$ is a nontrivial subspace, containing a nonzero vector, say $x \in \ker C_{XX}$. Suppose that a linear operator $A$ satisfies  $AC_{XX}=C_{XY}$ and $Ax = 0$ for all $x \in \ker C_{XX}$. From Corollary \ref{cor0}, we know that $A$ is an OLPO. Consider another linear operator, $\widetilde{A}$, which agrees with $A$ on $[\ker C_{XX}]^\perp$, but may not satisfy $\widetilde{A} x=0$  for $x \in \ker C_{XX}$. Corollary \ref{cor0} implies that the behavior of an operator on $\ker C_{XX}$ is irrelevant in characterizing characterizing OLPOs, and hence $\widetilde{A}$ is also an OLPO. In this way, many different OLPOs can be defined in this case. When $A$ satisfying $AC_{XX}=C_{XY}$ is not uniquely identified, practitioners may restrict attention to estimating $A$ on $[\ker C_{XX}]^\perp$, or equivalently, assume that $A = 0$ on $\ker C_{XX}$ (see e.g., \citealp{Benatia2017}, Section 3). As will be discussed, the direction pursued in this paper is significantly different, and we do not explore this further.} 
\end{remark}

\begin{remark}\label{rem1} \normalfont
	Condition (\(b\)) in Proposition \ref{prop1} was earlier obtained as the requirement for $A \in \mathcal L_{\infty}$ to be an OLPO by Propositions 2.2-2.3 of \cite{BOSQ2007879}. Compared with this earlier result, Proposition \ref{prop1} not only provides more general results, such as the explicit expression of the gap between the minimal attainable MSPE and the MSPE associated with any \(B \in \mathcal L_{\infty}\), but it also employs a distinct approach. The result of \cite{BOSQ2007879} relies on the notion of a linearly closed subspace and an extension of the standard projection theorem, while Proposition \ref{prop1} is established by an algebraic proof based on the representation of cross-covariance operators in \cite{baker1973joint}. 
\end{remark}

\begin{remark}\label{rem2} \normalfont
	By invoking the Hahn-Banach extension theorem (Theorem 1.9.1 of \citealp{megginson1998}), we may assume that \(A\) satisfying \(AC_{XX}=C_{XY}\) is a unique continuous linear map defined on the closure of  \(\ran C_{XX}\), which is not equal to \(\mathcal H\)  if \(C_{XX}\) is not injective. Thus, if there exists another continuous linear operator \(\widetilde{A}\) which agrees with \(A\) on the closure of \(\ran C_{XX}\) but not on \([\ran C_{XX}]^\perp\), then \(\widetilde{A}\) also satisfies that \(\mathbb{E}\|Y_t-\widetilde{A}X_t\|^2 = \|C_{YY}- C_{YY}^{1/2}R_{XY}R_{XY}^\ast C_{YY}^{1/2}\|_{\SC}\).
\end{remark}
The results given in Proposition \ref{prop1} and Remark \ref{rem2} imply that, particularly when the predictive variable is function-valued, there may be multiple OLPOs that satisfy \eqref{eqmse}. Furthermore, even if a unique OLPO exists, it may not be consistently estimable; a more detailed discussion is given in Example \ref{ex1} below. This means that we are in a somewhat different situation from the previous simple univariate case discussed in Section \ref{secintro}, where we can estimate the OLP by consistently estimating the OLPO. 
\begin{example} \label{ex1} \normalfont Suppose that  \(\{Y_t,X_t\}_{t\geq 1}\) satisfies \eqref{eq003}, \(X_t\) has a positive covariance \(C_{XX}\), and \(\varepsilon_t\) is serially independent and also independent of \(X_s\) for any \(s\). In this case, \(\AO C_{XX}=C_{XY}\), making \(\AO\) an OLPO. Since \(C_{XX}\) is injective, any continuous linear operator \(\widetilde{A}\) agrees with $\AO$ on the closure of  $\ran C_{XX}$ also agrees with \(\AO\) on \(\mathcal H\) (see Remark \ref{rem2} and note that the closure of \(\ran C_{XX}\) is \(\mathcal H\) in this case).
	However, consistently estimating $\AO$ without any further assumptions is impossible. To illustrate this, we may consider the case where $\{f_j\}_{j\geq 1}$ in \eqref{eq003} are known for simplicity. Even in this simplified scenario, there are infinitely many unknown parameters $\{a_j\}_{j \geq 1}$ to be estimated from only $T$ samples, necessitating additional assumptions on $\{a_j\}_{j\geq 1}$ for consistent estimation. 
\end{example}

\begin{remark}\label{remadd4} \revlab{r2_major1}\normalfont 
	\commRV{From the observation in Corollary \ref{cor0}, we know that an OLPO \( A \) satisfies \( C_{XX}A^\ast = C_{XY}^\ast \), and thus one might consider defining the unique OLPO \( A \) as the adjoint of \( C_{XX}^\dag C_{XY}^\ast \) by utilizing a generalized inverse \( C_{XX}^\dag \) (of \( C_{XX} \)), which uniquely exists. A standard notion of a generalized inverse in a Hilbert space setting may be the Moore--Penrose inverse; see \citet[Definition 2.2]{Engl2000} and also, e.g., \cite{BS2018} for its use in time series analysis in a Hilbert space. However, the existence of the Moore--Penrose inverse requires \( C_{XX} \) to have a closed range, but unique identification via this approach is not achieved; this is because any covariance operator is compact and hence cannot have a closed range unless it is of finite rank (see, e.g., \citealp{megginson1998}, Proposition 3.4.6). More general concepts of generalized inverses, which can even be defined in a Banach space setting, may also be considered, as in, e.g., \citet[Section 2.3]{seo_2023_fred} for more general function-valued time series, but they also require $C_{XX}$ to have a closed range.
	} 
\end{remark}

\begin{remark} \label{remadd3}\normalfont 
	\commRV{
		Given that \( A \) is a possibly non-unique solution to the linear operator equation \( C_{XX}A^\ast = C_{XY}^\ast \) in the Banach space of Shatten $p$-class operators for $p \in \mathbb{N}$, it may be possible to define its minimum norm solution $A^\dag$ satisfying $C_{XX} A^\dag =  C_{XY}^\ast$ and $\|A^\dag\| = \inf\{\|\tilde{A}\|: C_{XX} \tilde{A} = C_{XY}^\ast\}$, which exists under appropriate mathematical conditions, as the OLPO. It may also be possible to define an regularized solution $A^\dag_\alpha$ which approaches $A^\dag$, and characterize the convergence rate by extending the existing results (see e.g., \citealp{Schuster2012}, Chapter 3).  However, this is beyond the scope of the present paper, and we leave it for future work, with thanks to an anonymous referee for this observation.} 
\end{remark}

	\section{Estimation} \label{secest1} 
	
	We observed that in a general Hilbert space setting, there can be not only multiple OLPOs but also instances where, even if a unique OLPO exists, consistent estimation of it is impossible. 	Nevertheless, under mild conditions, we may construct a predictor using a standard post dimension-reduction estimator in such a way that the associated empirical MSPE converges to the minimum MPSE as in the simple univariate case. 
	To propose such a predictor, we let  $$\widehat{C}_{XY}=T^{-1}\sum_{t=1}^T X_t\otimes Y_t, \quad \widehat{C}_{XX}=T^{-1}\sum_{t=1}^T X_t\otimes X_t.$$	It is well known that $\widehat{C}_{XX}$ allows the spectral representation $\widehat{C}_{XX}=\sum_{j\geq1} \hat{\lambda}_j \hat{v_j} \otimes \hat{v}_j$ for the eigenvalues $\{\hat{\lambda}_j\}_{j \geq 1}$ and the corresponding eigenvectors $\{\hat{v}_j\}_{j \geq 1}$  (see, e.g., \citealp{Bosq2000}, p.\ 35). Let  
	\begin{align} \label{eqsection2}
		\widehat{C}_{XX,k_{T}} = \sum_{j=1}^{k_{T}} \hat{\lambda}_j \hat{v}_j \otimes \hat{v}_j, \quad 	\widehat{C}_{XX,k_{T}}^{-1} = \sum_{j=1}^{k_{T}} \hat{\lambda}_j^{-1} \hat{v}_j \otimes \hat{v}_j,
	\end{align}
	where \(k_T\) is an integer satisfying the following assumption: below, we let \(a_1 \wedge a_2 = \min\{a_1,a_2\}\) for \(a_1,a_2\in\mathbb{R}\) and assume that \(\max_{j\geq 1}\{j: E_j\} = 1\) if the condition \(E_j\) is not satisfied for all \(j\geq 1\).  
	\begin{assumption}[Elbow-like rule] \label{assumpelbow}
		$k_T$ in \eqref{eqsection2} is given by $${k}_T = \max_{j\geq 1} \left\{j: \hat{\lambda}_{j} \geq \hat{\lambda}_{j+1} + \tau_T\right\} \wedge \upsilon_{T}^{-1},$$ 
			where \(\tau_T\) and $\upsilon_{T}$ are user-specific choices of positive constants decaying to $0$ as \(T\to \infty\), and both $\tau_T^{-1}$ and $\upsilon_{T}^{-1}$ are bounded above by $\gamma_0 T^{\gamma_1}$ for some $\gamma_0>0$ and $\gamma_1 \in (0,1/2)$. 
		\end{assumption}
		\(\widehat{C}_{XX,k_{T}}^{-1}\) in \eqref{eqsection2} is understood as the inverse of \(\widehat{C}_{XX}\) viewed as a map acting on the restricted subspace \(\spn\{\hat{v}_j\}_{j=1}^{k_T}\) and \(k_T\) in Assumption \ref{assumpelbow} is a random integer by its construction (see Remark \ref{rem3} below). \revlab{r1_major1}\commRV{The quantity $\upsilon_T$ in the above assumption plays the role of an upper bound for the truncation parameter $k_T$, and by including \(\upsilon_T\) in Assumption~\ref{assumpelbow}, we ensure that \(k_T\) becomes \(o_p(T^{1/2})\). As will be discussed in detail later, we construct our estimator based on $\widehat{C}_{XX,k_{T}}^{-1}$, which becomes unstable as $k_T$ gets larger. Accordingly, Assumption~\ref{assumpelbow} imposes a general growth-rate condition on $k_T$ via $\upsilon_T$, thereby facilitating the proofs of our main theoretical results.} Even if the choice of \(k_T\) depends on various contexts requiring a  regularized inverse of \(\widehat{C}_{XX}\), it is commonly set to a much smaller number than \(T\), and thus this condition does not impose any practical restrictions. A practically more meaningful decision is made by the first component, \(\max_{j\geq 1} \{j:\hat{\lambda}_{j} \geq \hat{\lambda}_{j+1} + \tau_T\}\) in Assumption \ref{assumpelbow}.  Firstly, given that \(\hat{\lambda}_{k_T+1} > 0\), this condition implies that \(\hat{\lambda}_{k_T}^{-1} < \tau_T^{-1}\), and thus \(\|\widehat{C}_{XX,k_T}^{-1}\|_{\infty}\leq \tau_T^{-1}\), where \(\tau_{T}^{-1}\) diverges slowly compared to \(T\); clearly, this is one of the essential requirements for \(k_T\) (as the rank of a regularized inverse of $\widehat{C}_{XX}$) to satisfy in the literature employing similar regularized inverses.  Secondly, \(k_T\) is determined near the point where the gap \(\hat{\lambda}_{j} - \hat{\lambda}_{j+1}\) is no longer smaller than a specified threshold \(\tau_T\) for the last time. This approach is, in fact, analogous to the standard elbow rule used to determine the number of principal components in multivariate analysis based on the scree plot.
		Thus, even if Assumption \ref{assumpelbow} details some specific mathematical requirements necessary for our asymptotic analysis, these seem to closely align with existing practical rules for selecting \(k_T\), commonly employed in current practice; for example, see \cite{chang2021stock}. 
			
			Using the regularized inverse $\widehat{C}_{XX,k_T}^{-1}$, the proposed predictor is constructed as follows:   
			\begin{align} \label{eqpredictor}
				\widehat{Y}_t = \widehat{A}X_t, \quad \text{where}\quad  \widehat{A}(\cdot) = \widehat{C}_{XY}\widehat{C}_{XX,k_{T}}^{-1}(\cdot) = \frac{1}{T} \sum_{t=1}^T \sum_{j=1}^{k_T} \hat{\lambda}_j^{-1} \langle \hat{v}_j,\cdot \rangle \langle \hat{v}_j,X_t \rangle Y_t.
			\end{align}
			The above predictor is standard in the literature on functional data analysis as the least squares predictor of \(Y_t\) given the \textit{projection} of \(X_t\) onto the space spanned by the eigenvectors corresponding to the first $k_T$ largest eigenvalues; a similar estimator was earlier considered by \cite{Park2012397}. The predictor described in \eqref{eqpredictor} and its modifications (such as those that will be considered in Section \ref{sec_reduction}) have been widely studied in the literature and adapted to various contexts; see e.g.,  \cite{Bosq2000}, \cite{Yao2005}, \cite{aue2015prediction}, \cite{aue2017estimating},   \cite{klepsch2017prediction},  \cite{luo2017function} and \cite{Zhang2018}. 
			
			\begin{remark} \label{rem3}\normalfont
				A significant difference in \(\widehat{A}\) compared to most of the existing estimators, lies in the choice of \(k_T\). In many earlier articles, \(k_T\) is directly chosen by researchers and thus regarded as deterministic. However, as pointed out by \cite{seong2021functional}, even in this case, it is generally not recommended to choose \(k_T\) arbitrarily without taking the eigenvalues \(\hat{\lambda}_j\) into account. Therefore, it is natural to view \(k_T\) as random from a practical point of view.
			\end{remark}
			
			\begin{remark}\label{remadd1}\normalfont \revlab{r2_major4}
		\commRV{As in standard practice for predicting or forecasting functional time series, one may use $k_T$ that minimizes a prediction-error-type criterion, which is typically expressed as a function of $k_T$ (see, e.g., \citealp{aue2015prediction}). Alternatively, practitioners may prefer to minimize out-of-sample forecasting errors (see, e.g., \citealp{shang2019robust}). These approaches determine $k_T$ in a data-driven manner based on prediction accuracy. The requirements in Assumption~\ref{assumpelbow} can be incorporated into this criterion-based choice of $k_T$. For example, let $\tau_{T,1}, \ldots, \tau_{T,M}$ (resp.\ $\upsilon_{T,1}, \ldots, \upsilon_{T,M}$) be candidates for $\tau_T$ (resp.\ $\upsilon_T$), such that
		$c_{\min} T^{\gamma_1} < \tau_{T,1} < \cdots < \tau_{T,M} < c_{\max} T^{\gamma_1}$ (resp.\ $c_{\min} T^{\gamma_1} < \upsilon_{T,1} < \cdots < \upsilon_{T,M} < c_{\max} T^{\gamma_1}$), where $\gamma_1 \in (0,1/2)$ as required by Assumption~\ref{assumpelbow}, $0 < c_{\min} < c_{\max}$, and $M>0$. We then compute the prediction criteria for each pair $(\tau_{T,i}, \upsilon_{T,j})$ and identify the optimal pair $(\tau_{T,i^\ast}, \upsilon_{T,j^\ast})$ to serve as $\tau_T$ and $\upsilon_T$. This choice is justified both as a practical means of improving prediction accuracy and also ensures that $\widehat{A}$ is asymptotically OLPO, as required in this paper.}
			\end{remark}

			To establish consistency of \(\widehat{A}\) in \eqref{eqpredictor}, certain assumptions have been employed in the aforementioned literature, particularly concerning the eigenstructure of \(C_{XX}\) and the unique identification of the target estimator \(A\). However, as illustrated by Example \ref{ex1} and Remark \ref{rem2}, these assumptions are not guaranteed to hold even in cases where there exists a well defined OLPO. 
			Thus, in this paper, we do not make such assumptions for consistency, but allow \(\widehat{A}\) to potentially be inconsistent. Our main result in this section is that, despite not relying on the typical assumptions for consistency, the predictor \(\widehat{Y}_t\) asymptotically minimizes the MSPE, which only requires mild conditions on the sample (cross-)covariance operators.	
			
			For the subsequent discussion, we define the following: for any \(B \in \mathcal L_{\infty}\),
			\begin{equation*}
				\Sigma(B) =\frac{1}{T}\sum_{t=1}^T (Y_t - BX_t) \otimes (Y_t - BX_t).  
			\end{equation*}
			Note that \(\|\Sigma(B)\|_{\SC}\) is equivalent to the empirical MSPE, given by \(T^{-1} \sum_{t=1}^T \|Y_t-BX_t\|^2\). We then define an asymptotically OLPO as a random bounded linear operator producing a predictor that asymptotically minimizes the MSPE in the following sense:
			\begin{definition}
				Any random bounded linear operator \(\widehat{B}\) is called an asymptotically OLPO if
				\begin{equation*}
					\|\Sigma(\widehat{B})\|_{\SC} \to_p \Sigma_{\min} \text{\quad as \(T\to \infty\)},
				\end{equation*}
				where \(\Sigma_{\min} = \|C_{YY}- C_{YY}^{1/2}R_{XY}R_{XY}^\ast C_{YY}^{1/2}\|_{\SC}\), which is the minimum MSPE that can be achieved by a linear predictor, as defined in Proposition \ref{prop1}.
			\end{definition}
			We will employ the following assumption: 
			\begin{assumption}[Standard rate of convergence]\label{assum1}  $\|\widehat{C}_{YY}-C_{YY}\|_{\infty}$, $\|\widehat{C}_{XX}-C_{XX}\|_{\infty}$, and $\|\widehat{C}_{XY}-C_{XY}\|_{\infty}$ are $O_p(T^{-1/2})$.
			\end{assumption} 
			We observe that \(\{Y_t\otimes Y_t - C_{YY}\}_{t\geq 1}\),   \(\{X_t\otimes X_t - C_{XX}\}_{t\geq 1}\), and  \(\{X_t\otimes Y_t - C_{XY}\}_{t\geq 1}\) are stationary sequences of Schatten 2-class operators. These operator-valued sequences may also be understood as stationary sequences in a separable Hilbert space  (\citealp{Bosq2000}, p.\ 34). Then, Assumption \ref{assum1} is satisfied under some non-restrictive regularity conditions; see e.g., Theorems 2.16-2.18 of \cite{Bosq2000} concerning the central limit theorems for Hilbert-valued random elements. Given that we are dealing with a weakly stationary sequence \(\{Y_t,X_t\}_{t\geq 1}\), Assumption \ref{assum1} appears to be standard. Furthermore, it can be relaxed to a weaker requirement by imposing stricter conditions on \(\tau_T\) and \(\upsilon_T\) in Assumption \ref{assumpelbow}; this will be detailed in Section \ref{sec_general}.

				We now present the main result of this paper. 
				\begin{theorem}\label{thmblp} Under Assumptions \ref{assumpelbow}-\ref{assum1}, $\widehat{A}$ is an asymptotically OLPO, i.e.,  \(\|\Sigma(\widehat{A})\|_{\SC} \to_p \Sigma_{\min}\).
				\end{theorem}
				As stated, an appropriate growth rate of \(k_T\), detailed in Assumption \ref{assumpelbow}, and the standard rate of convergence of the sample (cross-)covariance operators (Assumption \ref{assum1}) are all that we need in order to demonstrate that the proposed predictor in \eqref{eqpredictor} is an asymptotically OLPO. As discussed earlier, since the choice rule in Assumption \ref{assumpelbow} is practically similar to existing rules for selecting $k_T$, a naive use of the simple post-dimension reduction estimator \(\widehat{A}\) in \eqref{eqpredictor} without the usual assumptions for the unique identification and/or consistency, which are widely employed but challenging to verity, can still be justified as a way to asymptotically minimize the MSPE (see Section \ref{sec_misspecified}).  Some discussions and extensions on the above theorem are given in the next section.  We also consider the case where \(\widehat{A}\) is replaced by another estimator employing a different regularization scheme (Sections~\ref{sec_noreduction}-\ref{sec_reduction}).  
				
				\begin{remark} \normalfont \label{remadd2}
					\commRV{
				Beyond our focus on optimal linear prediction in \(\mathcal H\) under minimal conditions, one may further investigate the statistical properties of the proposed estimator $\widehat{A}$. Since the estimator is constructed using the elbow-like rule (Assumption~\ref{assumpelbow}), it can be viewed as a special case of spectral cut-off regularization (see, e.g., \citealp{Carrasco2007}). Following the approach of \cite{Benatia2017} for a related estimator based on ridge- and Tikhonov-type regularization methods, and noting that their framework can be extended to more general regularization schemes, one may, under an appropriate source condition, establish consistency of the estimator and its convergence rate, as well as provide a detailed analysis of the regularization bias. We leave this investigation for future research and gratefully acknowledge an anonymous referee for highlighting this insight.	
					}
				\end{remark}
				\section{Discussions and extensions} \label{sec_disext}
				\subsection{A more general result} \label{sec_general}
				In Assumption \ref{assum1}, we assumed that the sample (cross-)covariances converge to the population counterparts with \(\sqrt{T}\)-rate. 
				However, the exact \(\sqrt{T}\)-rate is not mandatory for 
				the desired result and can be relaxed if we make appropriate adjustments on \(\tau_T\) and \(\upsilon_T \) in Assumption \ref{assumpelbow} as follows:
				\begin{corollary} \label{cor1}
					Suppose that, for \(\beta \in (0,1/2]\), \(\|\widehat{C}_{YY}-C_{YY}\|_{\infty}\), \(\|\widehat{C}_{XX}-C_{XX}\|_{\infty}\), and \(\|\widehat{C}_{XY}-C_{XY}\|_{\infty}\) are  \(O_p(T^{-\beta})\). If Assumption \ref{assumpelbow} holds for \(\gamma_1 \in (0,\beta)\), then \(\widehat{A}\) is an asymptotically OLPO. 
				\end{corollary}
				That is, \(\widehat{A}\) is an asymptotically OLPO under weaker assumptions on the sample (cross-)covariance operators if we impose stricter conditions on the decay rates of \(\tau_T\) and \(\upsilon_T\). \(\tau_T\) and \(\upsilon_T\) are user-specific choices dependent on \(T\), so researchers can easily manipulate their decay rates. Corollary \ref{cor1}, thus, tells us that we can make 
				\(\widehat{A}\) an asymptotically OLPO under more general scenarios by simply reducing the decay rates of \(\tau_T\) and \(\upsilon_T\). 
				

				\subsection{Misspecified functional linear models and OLPO} \label{sec_misspecified}
				Consider the standard functional linear model 
				\begin{align} \label{eqflm}
					Y_t = AX_t + \varepsilon_t, \quad \mathbb{E}[X_t \otimes \varepsilon_t] = 0,
				\end{align}
				where $A$ is typically assumed to satisfy the following two conditions: (i) $A$ is  Hilbert-Schmidt (i.e., $A \in \SCC$) and (ii) $A$ is uniquely identified (in $\SCC$). A common assumption employed for the unique identification is that $\ran C_{XX}$ is dense in $\mathcal H$ (see Remark \ref{rem2}). 
				Under additional technical assumptions on the eigenstructure of $C_{XX}$, such as those on the spectral gap \((\lambda_j-\lambda_{j-1})\) as in \cite{Hall2007}, the proposed estimator \(\widehat{A}\) in \eqref{eqpredictor} turns out to be consistent if \(k_T\) grows appropriately. Of course, some alternative estimators, such as the least squares estimator with Tikhonov regularization (see, e.g., \citealp{Benatia2017}), do not require any assumptions on the spectral gap, which is a well known advantage of such methods. However, they still require condition (i), the Hilbert-Schmidt property of \( A \), with some additional regularity assumptions on \( A \), and also impose assumptions for condition (ii) when discussing the consistency of those estimators. 
				
				It is crucial to have conditions (i) and (ii) together with the model \eqref{eqflm}, since a violation of either of these  can easily lead to inconsistency (see Example \ref{ex1}). 
				While these requirements are standard for estimation, they exclude many natural data generating mechanisms in $\mathcal H$ (see Examples~\ref{ex1}-\ref{ex2}).  Consequently, the model may suffer from misspecification issues. However, even in such cases, our results show that \(\widehat{A}\) attains the minimum MSPE asymptotically if \(k_T\) grows at an appropriate rate, and thus affirm the potential use of the standard post-dimension reduction estimator in practice without a careful examination of various technical conditions.   
				
				\begin{example}\label{ex2} \normalfont Similar to the example given in Section 5 of \cite{BSS2017}, suppose that \(X_t = Y_{t-1}\) and \(Y_t\) satisfies the functional AR(1) law of motion: \(Y_t=A Y_{t-1} + \varepsilon_t\)  with \(A = \sum_{j=1}^\infty a_j w_j \otimes w_j\) for some orthonormal basis \(\{w_j\}_{j\geq 1}\) and iid sequence \(\{\varepsilon_t\}_{t\in \mathbb{Z}}\) (see also \citealp{aue2015prediction}). This leads to the following pointwise AR(1) model:  \(\langle Y_{t}, w_j \rangle = a_j\langle Y_{t-1}, w_j \rangle + \langle \varepsilon_t, w_j \rangle\). This time series is guaranteed to be stationary if \(\sup_{j}|a_j|<1\); this can be demonstrated with only a slight and obvious modification of Theorem 3.1 of \cite{Bosq2000} or Proposition 3.2 of \cite{seo_2022}. On the other hand, for \(A\) to be a Hilbert-Schmidt operator, a much more stringent condition \(\sum_{j=1}^{\infty} |a_j|^2 <\infty\) is required, and, as a consequence, \(\langle Y_{t}, w_j \rangle\) and \(\langle Y_{t-1}, w_j \rangle\) need to be nearly uncorrelated for large \(j\) while they can be arbitrarily correlated under the aforementioned condition for weak stationarity. 
				\end{example}
				\subsection{Requirement of sufficient dimension reduction}\label{sec_noreduction}

				In our proof of Theorem \ref{thmblp} (resp.\ Corollary \ref{cor1}), it is crucial to have a regularized inverse of \(\widehat{C}_{XX}\), denoted as \(\widehat{C}_{XX,k_{T}}^{-1}\), whose rank $k_T$ grows at a sufficiently slower rate than $T$; see e.g.,  \eqref{eqadd4} and \eqref{eqadd5} in Appendix \ref{sec_math}.  
				Due to this requirement, our arguments for proving the main results (Theorem \ref{thmblp} and Corollary \ref{cor1}) are not straightforwardly extended to other popular estimators without dimension reduction, such as least squares-type estimators with ridge or Tikhonov regularization (see, e.g., \citealp{Benatia2017}). \revlab{r1_major3}\commRV{Importantly, these alternative estimators can still attain prediction performance on par with Theorem \ref{thmblp} and Corollary \ref{cor1}, even when the usual Hilbert–Schmidtness and unique identification assumptions on the OLPO $A$ are relaxed, but a systematic study of these methods is left for future work.} 

				
				\subsection{Dimension reduction of the target variable}\label{sec_reduction}
				The proposed estimator \(\widehat{A}\) is commonly used as a standard estimator of the functional linear model \eqref{eqflm}. Note that, in our construction of \(\widehat{A}\), the target variable \(Y_t\) is used as is, without any dimension reduction. 
				In the literature, estimators similar to \(\widehat{A}\) in \eqref{eqpredictor}, but where \(Y_t\) is replaced with its version obtained through dimension reduction, also appear to be popular and are used in practice (see e.g., \citealp{Yao2005}). As noted in recent articles, dimension reduction of the target variable not only is generally non-essential for establishing certain key asymptotic properties (such as consistency) of the estimator but may also lead to a less optimal estimator (see Remark 1 of  \citealp{imaizumi2018}). 
				
				Consider the following predictor and estimator, constructed using a version of \(Y_t\) with reduced dimension: 
				\begin{equation*}
					\widetilde{Y}_t = \widetilde{A}X_t, \quad \text{where}\quad \widetilde{A}(\cdot) = \widehat{\Pi}_{Y,\ell_T} \widehat{C}_{XY}\widehat{C}_{XX,k_{T}}^{-1}(\cdot) = \frac{1}{T} \sum_{t=1}^T \sum_{j=1}^{k_T} \hat{\lambda}_j^{-1} \langle \hat{v}_j,\cdot \rangle \langle \hat{v}_j,X_t \rangle \widehat{\Pi}_{Y,\ell_T} Y_t
				\end{equation*}
				and $\widehat{\Pi}_{Y,\ell_T} = \sum_{j=1}^{\ell_T} \hat{w}_j \otimes \hat{w}_j$ 
				for some orthonormal basis \(\{\hat{w}_j\}_{j\geq 1}\) and  \(\ell_T\) growing as $T$ increases; often, \(\hat{w}_j\) is set to the eigenvector of \(\widehat{C}_{YY}\) or  \(\widehat{C}_{XX}\) corresponding to the \(j\)-th largest eigenvalue, but our subsequent analysis is not restricted to these specific cases. 
				Even if the additional dimension reduction applied to \(Y_t\) introduces some complications in our theoretical analysis, it can also be shown that \(\widetilde{A}\) is an asymptotically OLPO under an additional mild condition. 
				\begin{corollary}\label{cor2}
					Suppose that Assumptions \ref{assumpelbow}-\ref{assum1} hold and there exists a sequence $m_T$ tending to infinity as $T \to \infty$ such that  $m_T \|\widehat{\Pi}_{Y,\ell_T} C_{YY}^{1/2} - C_{YY}^{1/2}\|_{\infty} \to_p 0$ and $m_T \leq \ell_T$ eventually. Then, $\widetilde{A}$ is an asymptotically OLPO. 
				\end{corollary}
				Given that \(\widehat{\Pi}_{Y,\ell_T}\) is the orthogonal projection with a growing rank, the condition given in Corollary \ref{cor2} becomes easier to be satisfied if \(\ell_T\) grows more rapidly. In this sense,, the scenario in Theorem \ref{thmblp} can be seen as the limiting case where \(\widehat{\Pi}_{Y,\ell_T} = I\), indicating no dimension reduction applied to \(Y_t\). Corollary \ref{cor2} implies that estimators obtained by reducing the dimensionality of \(Y_t\) tend to be asymptotically OLPOs under non-restrictive conditions. Given that  \(C_{YY}^{1/2}\)and \(C_{YY}\) share the same eigenvectors, one may conjecture that satisfying the requirement in Corollary \ref{cor2} could be easier if \(\widehat{w}_j\) is an eigenvector of \(\widehat{C}_{YY}\). Theoretical justification of this may require further assumptions on the eigenstructure of $C_{YY}$. Given the focus on optimal linear prediction in \(\mathcal H\) under minimal conditions, we do not pursue this direction and leave it for future study.

					\section{Simulation} \label{secest2} 
					We provide simulation evidence for our theoretical findings, focusing on cases that have not been sufficiently explored in the literature and where consistent estimation of the OLPO is impossible. In all simulation experiments, the number of replications is 1000. 
					
					Let \(\{f_j\}_{j \geq 1}\) be the Fourier basis of \(L^2[0,1]\),  the Hilbert space of square-integrable functions on \([0,1]\), 
					i.e., for \(x \in [0,1]\),  \(f_1(x) = 1\), and for \(j\geq 2\),   \(f_j = \sqrt{2}\sin(2\pi jx)\) if \(j\) is even, and \(f_j = \sqrt{2}\cos(2\pi jx)\) if \(j\) is odd.    We define  \(X_t\) and \(Y_t\) as follows: for some real numbers \(\{a_j\}_{j=1}^{101}\),
					\begin{equation*} 
						Y_t = A X_t + \varepsilon_t, \quad X_t = \sum_{j=1}^{101} a_j \langle X_{t-1}, f_j \rangle f_j + e_t, 
					\end{equation*}
					where \(\{e_t\}_{t\geq 1}\) and \(\{\varepsilon_t\}_{t\geq 1}\) are assumed to be mutually and serially independent sequences of random elements.  In this simulation setup, the minimal MSPE that can be achieved by a linear operator equals the Schatten 1-norm (trace norm) of the covariance operator \(C_{\varepsilon}\) of \(\varepsilon_t\). 
					
					We conducted experiments in three different cases. In the first case (referred to as Case BB), \(e_t\)  and  \(\varepsilon_t\) are set to independent realizations of the standard Brownian bridge, and in the second case (referred to as Case CBM), they are set to independent realizations of the centered Brownian motion. In these first two cases, the \(j\)-th largest eigenvalue of \(C_{\varepsilon}\) is given by \(\pi^{-2}j^{-2}\); see \citet[p.\ 86]{JaimezBonnet} and \citet[p.\ 1465]{karol2008small}. From a well known result on the Riemann zeta function (see e.g., \citealp{kalman2012another}),  we find that these eigenvalues add up to \(1/6\), which is the minimal achievable MSPE. In the last case (referred to as Case BM), \(e_t\) and \(\varepsilon_t\) are set as realizations of the standard Brownian motion multiplied by a constant, which is  chosen to ensure that the minimal MSPE in this scenario matches that in the previous two cases. 
					
					
					We will subsequently consider the following four models, depending on the specification of \(a_j\) and \(A\) for generating \(X_t\) and \(Y_t\): {for all $j\geq 1$},
					\begin{align*}
						\text{Model\,1: }& a_j\sim_{\text{iid}} U(-0.1,0.25) \,\,  \text{and} \,\,  A f_j = b_0 f_j,\\
						\text{Model\,2: }& a_j\sim_{\text{iid}} U(-0.1,0.25) \,\,  \text{and} \,\, A f_j = b_j f_j 1\{j\leq 100\} +  f_j 1\{j > 100\},   \\
						\text{Model\,3: }& a_j \sim_{\text{iid}} U(-0.1,0.75) \,\,  \text{and} \,\,  A f_j = b_0 f_j, \\
						\text{Model\,4: }& a_j \sim_{\text{iid}} U(-0.1,0.75) \,\,  \text{and} \,\,  A f_j = b_j f_j 1\{j\leq 100\} +  f_j 1\{j > 100\},
					\end{align*}
					where \(b_0 \sim U[-2.5,2.5]\) and \(b_j \sim_{\text{iid}} U[-2.5,2.5]\) for $j \geq 1$. Note that the parameters \(a_j\), \(b_0\) and \(b_j\) are generated differently in each simulation run; this allows us to assess  the average performance of the proposed predictor across various parameter choices. In many empirical examples involving dependent sequences of functions \(X_t\), it is often expected that \(\langle X_t,v \rangle\) for any \(v \in \mathcal H\) exhibits a positive lag-one autocorrelation, so we let \(a_j\) tend to take positive values more frequently in our simulation settings. In any of the above cases, \(A\) is non-compact and hence cannot be consistently estimated without prior knowledge on the structure of \(\{A f_j\}_{j \geq 1}\), which is as in the operator considered in Example \ref{ex1}.  		
					
					We set \(\tau_T = 0.01 \|\widehat{C}_{XX}\|_{\SC} T^{\gamma}\) and \(\upsilon_T = 0.5 T^{\gamma}\) for some \(\gamma > 0\), where note that \(\tau_T\) is designed to reflect the scale of $X_t$, as proposed by \cite{seong2021functional} in a similar context. We then computed the empirical MSPE associated with \(\widehat{A}\), introduced in \eqref{eqpredictor}. Table \ref{tab1} reports the excess MSPE (the empirical MSPE minus $1/6$) for each of the considered cases. 
					As expected from our main theoretical results, the excess MSPE associated with \(\widehat{A}\) approaches zero as the sample size \(T\) increases, even if \(\widehat{A}\) is not consistent. Notably, in Case BM, the excess MSPE tends to be significantly smaller than in the other two cases (Case BB and Case CMB); even with a moderately large number of observations, the empirical MSPE in Case BM tends to be close to the minimal MSPE. This suggests that the performance of the proposed predictor significantly depends on the specification of \(\varepsilon_t\). Overall, the simulation results reported in Table \ref{tab1} support our theoretical finding in Section \ref{secest1}. We also experimented with an alternative estimator \(\widetilde{A}\) which is introduced in Section \ref{sec_reduction} and obtained qualitatively similar supporting evidence; some of the simulation results are reported in Appendix \ref{sec_addmonte} of the appendix. 
					
					\begin{remark}\label{remrevadd}\normalfont 
			\revlab{r1_major2}	\commRV{In the  considered simulation design, where \( X_t \) is not generated from a PCA-like structure, the proposed estimator in this paper may not perform as well as alternative methods based on soft thresholding methods such as ridge or Tikhonov regularization (see, e.g., \citealp{Carrasco2012}; \citealp{Florence2015}; \citealp{Benatia2017}; \citealp{babii2022high}), which may give better finite-sample properties. As discussed in Section~\ref{sec_noreduction}, our approach does not extend to these alternative estimators, which may be more attractive in certain contexts. This points to a direction for future research on this topic. We are grateful to an anonymous referee for bringing this insight to our attention.} 
				     \end{remark}

					\begin{table}[t!] 
						\small
						\caption{Excess MSPE of the proposed predictor}  \label{tab1}
						\renewcommand*{\arraystretch}{1} \vspace{-0.1em}
						\begin{minipage}{\linewidth}
							\centering  
							\subcaptionbox{Case BB	\vspace{-0.55em}	}{%
								\begin{tabular*}	{1\linewidth}{@{\extracolsep{\fill}}c@{\hspace{1\tabcolsep}}| ccccc@{\hspace{1\tabcolsep}}|ccccc}
									\hline\hline
									& &\multicolumn{3}{c}{$\gamma= 0.475$}& & &\multicolumn{3}{c}{$\gamma= 0.45$}&\\ \hline 
									\hspace{-0.25cm}	$T$ & $50$ & $100$ & $200$ &{$400$} & {$800$} & $50$ & $100$ & $200$ &{$400$} & {$800$} \\
									\hline 
									\hspace{-0.0cm}	Model 1&0.043& 0.035& 0.023& 0.018& 0.012& 0.072& 0.051& 0.031& 0.022& 0.014 \\
									\hspace{-0.0cm}	Model 2 & 0.041& 0.033& 0.022& 0.017& 0.011&   0.068 &0.049& 0.029 &0.021 &0.013 \\
									\hspace{-0.0cm}	Model 3 & 0.056& 0.046& 0.031& 0.023 &0.016& 0.094 &0.066& 0.040& 0.028& 0.019 \\
									\hspace{-0.0cm}	Model 4 & 0.054& 0.044 &0.029& 0.022 &0.015& 0.089& 0.063 &0.038& 0.027& 0.018 \\
									\hline\hline 
								\end{tabular*}
							}
						\end{minipage} \\[0.4em]
						
						\begin{minipage}{\textwidth}
							\centering
							\subcaptionbox{Case CBM \vspace{-0.55em}	}{%
								\begin{tabular*}{1\linewidth}{@{\extracolsep{\fill}}c@{\hspace{1\tabcolsep}}| ccccc@{\hspace{1\tabcolsep}}|ccccc}
									\hline\hline
									& &\multicolumn{3}{c}{$\gamma= 0.475$}& & &\multicolumn{3}{c}{$\gamma= 0.45$}&\\ \hline 
									\hspace{-0.25cm}	$T$ & $50$ & $100$ & $200$ &{$400$} & {$800$} & $50$ & $100$ & $200$ &{$400$} & {$800$} \\
									\hline 
									\hspace{-0.0cm}	Model 1& 0.045& 0.036& 0.024& 0.018& 0.012 & 0.074 &0.052& 0.031& 0.022 &0.014 \\
									\hspace{-0.0cm}	Model 2 & 0.041& 0.034& 0.022 &0.017 &0.011&   0.069& 0.049& 0.029& 0.020& 0.013 \\
									\hspace{-0.0cm}	Model 3 &  0.059& 0.048& 0.032& 0.024& 0.016 &0.098& 0.068& 0.041& 0.029& 0.019 \\
									\hspace{-0.0cm}	Model 4 &0.055 &0.045& 0.030& 0.022 &0.015&0.091 &0.064 &0.038& 0.027 &0.018 \\
									\hline\hline
								\end{tabular*}
							}
						\end{minipage}\\[0.4em]
						
						\begin{minipage}{\textwidth}
							\centering
							\subcaptionbox{Case BM \vspace{-0.55em}	}{%
								\begin{tabular*}{1\linewidth}{@{\extracolsep{\fill}}c@{\hspace{1\tabcolsep}}| ccccc@{\hspace{1\tabcolsep}}|ccccc}
									\hline\hline
									& &\multicolumn{3}{c}{$\gamma= 0.475$}& & &\multicolumn{3}{c}{$\gamma= 0.45$}&\\ \hline 
									\hspace{-0.25cm}		$T$ & $50$ & $100$ & $200$ &{$400$} & {$800$} & $50$ & $100$ & $200$ &{$400$} & {$800$} \\
									\hline 
									\hspace{-0.0cm}	Model 1&  0.011 &0.010& 0.006& 0.004& 0.004& 0.026& 0.017& 0.009& 0.006& 0.004 \\
									\hspace{-0.0cm}	Model 2 & 0.012 &0.010 &0.006& 0.004& 0.004& 0.027 &0.017 &0.009 &0.006 &0.004 \\
									\hspace{-0.0cm}	Model 3 & 0.018& 0.014& 0.009& 0.006& 0.005& 0.037& 0.024& 0.013& 0.008& 0.006 \\
									\hspace{-0.0cm}	Model 4 &0.018& 0.015& 0.009 &0.006& 0.005& 0.038& 0.024& 0.013& 0.008& 0.006 \\
									\hline\hline 
								\end{tabular*}
							}
						\end{minipage}
						
						\vspace{0.4em}
						\begin{minipage}{1\textwidth}
							{\footnotesize Notes: The excess MSPE is calculated as the empirical MSPE minus \(1/6\), where \(1/6\) represents the minimal achievable MSPE by a linear predictor. \(\tau_T = 0.01 \|\widehat{C}_{XX}\|_{\SC} T^{\gamma}\) and \(\upsilon_T = 0.5 T^{\gamma}\) for \(\gamma \in  \{0.45,0.475\}\). The reported MSPEs are approximately computed by (i) generating $X_t$ and $Y_t$ on a fine grid of $[0,1]$ with 200 equally spaced grid points, and then (ii) representing those with 100 cubic B-spline functions. The results exhibit little change with varying numbers of grid points and B-spline functions.} 
						\end{minipage}
						
					\end{table}

					\section{Concluding remarks} \label{secconclusion}	
					This paper studies linear prediction in a Hilbert space, demonstrating that, under mild conditions, the empirical MSPEs associated with standard post-dimension reduction estimators approach the minimal achievable MSPE. There is ample room for future research; for example, it would be intriguing to explore whether similar prediction results can be obtained from various alternatives or modifications of the simple post-dimension reduction estimators considered in the literature.
					
					\newpage 
					\appendix  
					\section*{Appendix}
					\section{Mathematical proofs}\label{sec_math}
					\subsection{Useful lemmas}
					\begin{lemma} \label{lema1}
						Let \(\Gamma\) be a nonnegative self-adjoint Schatten 1-class  operator. For any \(D \in \mathcal L_{\infty}\), the following hold.
						\begin{enumerate}[(i)]
							\item \label{lema1a}  $\Gamma^{1/2}DD^\ast \Gamma^{1/2}$ and $D \Gamma D^\ast$ are Schatten 1-class  operators and their Schatten 1-norms are bounded above by $\|D\|_{\infty}^2 \|\Gamma\|_{\SC}$.
							\item \label{lema1b} $\Gamma^{1/2}D$ and $D^\ast \Gamma^{1/2}$ are Schatten 2-class operators.
						\end{enumerate}
					\end{lemma}
					\begin{proof} 
						Let $\Gamma = \sum_{j=1}^\infty c_j w_j \otimes w_j$, where $c_1 \geq c_2\geq \ldots \geq 0$ and $c_j$ may be zero. By allocating a proper vector to each zero eigenvalue, we may assume that $\{w_j\}_{j\geq 1}$ is an orthonormal basis of $\mathcal H$. To show (i), we note that $\langle \Gamma^{1/2}DD^\ast \Gamma^{1/2} w_j,w_j\rangle = c_j \langle w_j, DD^\ast w_j \rangle \leq c_j \|D\|^2_{\infty}$, from which 	$\|\Gamma^{1/2}DD^\ast \Gamma^{1/2}\|_{\SC}\leq \|D\|_{\infty}^2 \|\Gamma\|_{\SC}$ is established. The desired result for $\|D\Gamma D^\ast\|_{\SC}$ is already well known, see e.g., \citet[p.\ 267]{Conway1994}. To show (ii), we observe that $\| \Gamma^{1/2}D\|_{\SCC}^2=\|D^\ast \Gamma^{1/2}\|_{\SCC}^2 \leq 
						\|D\|^2_{\infty} \sum_{j=1}^{\infty} \|\Gamma^{1/2}w_j\|^2 =\|D\|^2_{\infty} \sum_{j=1}^{\infty} c_j < \infty$, which establishes the desired result.
					\end{proof}
					\begin{lemma} \label{lema2}
						Let $\{\Gamma_j\}_{j\geq 1}$ be a sequence of random self-adjoint Schatten 1-class operators and let $\Gamma$ be a self-adjoint Schatten 1-class operator. Then, for any orthonormal basis $\{w_j\}_{j\geq 1}$ of $\mathcal H$ and $m_T\geq 0$, \begin{equation} \label{eqlem01}\|\Gamma_j-\Gamma\|_{\SC} = O_p(\|\Gamma_j\|_{\SC} - \|\Gamma\|_{\SC}) + O_p\left(m_T\|\Gamma_j-\Gamma\|_{\infty} + \sum_{j=m_T+1}^\infty \langle \Gamma w_j, w_j \rangle\right).
						\end{equation} 
						Moreover, $\|\Gamma_j-\Gamma\|_{\SC} \to_p 0$ if $\|\Gamma_j\|_{\SC} - \|\Gamma\|_{\SC}\to_p 0$ and $m_T\|\Gamma_j-\Gamma\|_{\infty} \to_p 0$ as $m_T\to \infty$ and $T\to \infty$.
					\end{lemma}
					\begin{proof}
						Equation \eqref{eqlem01} directly follows from Lemma 2 (and also Theorem 2) of  \cite{kubrusly1985convergence}. Moreover, if $\|\Gamma_j\|_{\SC} - \|\Gamma\|_{\SC}\to_p 0$ and $m_T\|\Gamma_j-\Gamma\|_{\infty} \to_p 0$ as $m_T\to \infty$ and $T\to \infty$, we find that $\|\Gamma_j-\Gamma\|_{\SC} = O_p(\sum_{j=m_T+1}^\infty \langle \Gamma w_j, w_j \rangle)$. Since $\Gamma$ is a self-adjoint Schatten 1-class operator, we have  $\sum_{j=1}^\infty \langle \Gamma w_j, w_j \rangle \to_p \|\Gamma\|_{\SC} < \infty$, from which we conclude that  $\sum_{j=m_T+1}^\infty \langle \Gamma w_j, w_j \rangle \to_p 0$ as $m_T \to \infty$.
					\end{proof}
					\subsection{Proofs of the theoretical results}
					\begin{proof}[Proof of Proposition \ref{prop1}]
						From Theorem 1 of \cite{baker1973joint}, we find that $C_{XY}$ allows the following representation: for a unique bounded linear operator $R_{XY}$ satisfying $\|R_{XY}\|_{\infty} \leq 1$,
						\begin{align*} 
							C_{XY}  = {C}_{YY}^{1/2}R_{XY} {C}_{XX}^{1/2}.
						\end{align*}
						Let $C_{\min}=C_{YY}-C_{YY}^{1/2}R_{XY} R_{XY}^\ast {C}_{YY}^{1/2}$, which is clearly self-adjoint. Moreover, $C_{\min}$ is a nonnegative Schatten 1-class operator. To see this, note that for any $w \in \mathcal H$,  
						\begin{align} 
							\langle C_{\min}w,w \rangle =\langle (C_{YY}-C_{YY}^{1/2}R_{XY} R_{XY}^\ast {C}_{YY}^{1/2})w,w \rangle = \|C_{YY}^{1/2}w\|^2- \|R_{XY}^\ast {C}_{YY}^{1/2}w\|^2 \geq 0 \label{eqpf001aaa}
						\end{align}
						(i.e., $C_{\min}$ is nonnegative), which is because $\|R_{XY}^\ast {C}_{YY}^{1/2}w\|^2 \leq \|R_{XY}^\ast\|_{\infty}^2 \|{C}_{YY}^{1/2}w\|^2 \leq  \|{C}_{YY}^{1/2}w\|^2$. Moreover, from \eqref{eqpf001aaa} and Lemma \ref{lema1}\eqref{lema1b}, we know that, for any orthonormal basis $\{w_j\}_{j\geq 1}$, $\|C_{\min}\|_{\SC} = \sum_{j=1}^\infty (\|C_{YY}^{1/2}w_j\|^2- \|R_{XY}^\ast {C}_{YY}^{1/2}w_j\|^2) < \infty$. 
						
						We then find the following holds for any $B\in \mathcal L_{\infty}$ and $w\in \mathcal H$:
						\begin{align}
							&	\langle(C_{YY} - C_{XY}B^\ast - B C_{XY}^\ast   +  BC_{XX}B^\ast) w,w\rangle \notag \\
							& = \|C_{YY}^{1/2}w\|^2 +\|C_{XX}^{1/2}B^\ast w\|^2 - 2 \langle  C_{XX}^{1/2}B^\ast w,R_{XY}^\ast C_{YY}^{1/2}w \rangle \notag  \\&=   \| C_{XX}^{1/2}B^\ast w- R_{XY}^\ast C_{YY}^{1/2}w\|^2 + \|C_{YY}^{1/2}w\|^2- \|R_{XY}^\ast C_{YY}^{1/2}w\|^2 \notag \\ &=   \| C_{XX}^{1/2}B^\ast w- R_{XY}^\ast C_{YY}^{1/2}w\|^2 + \langle C_{\min} w,w\rangle. \label{eqpf001}
						\end{align}
						We know from \eqref{eqpf001} that 
						\begin{align}
							\mathbb{E}\|Y_t-BX_t\|^2 &= 	\sum_{j=1}^\infty \langle(C_{YY} - C_{XY}B^\ast - B C_{XY}^\ast   +  BC_{XX}B^\ast) w_j,w_j\rangle\notag  \\
							&= \sum_{j=1}^\infty \| C_{XX}^{1/2}B^\ast w_j- R_{XY}^\ast C_{YY}^{1/2}w_j\|^2 + \|C_{\min}\|_{\SC}, 
							\label{eqpf001a}
						\end{align}
						where $\{w_j\}_{j\geq 1}$ is any orthonormal basis of $\mathcal H$.  
						Note that $C_{XX}^{1/2}B^\ast- R_{XY}^\ast C_{YY}^{1/2}$ is a Schatten 2-class operator (Lemma \ref{lema1}\eqref{lema1b}) and thus we find that $\sum_{j=1}^\infty \| C_{XX}^{1/2}B^\ast w_j- R_{XY}^\ast C_{YY}^{1/2}w_j\|^2 = \|C_{XX}^{1/2}B^\ast-R_{XY}^\ast C_{YY}^{1/2}\|_{\SCC}^2$. From this result combined with \eqref{eqpf001a}, the desired result immediately follows. 	
					\end{proof}   
					\begin{proof}[Proofs of Theorem \ref{thmblp} and Corollary \ref{cor1}]
						To accommodate more general cases, which are considered in Corollary \ref{cor1}, we hereafter assume that $\tau_T^{-1} = O_p(T^{\gamma_1})$ and $\upsilon_T^{-1} = O_p(T^{\gamma_1})$  for some $\gamma_1 \in (0,\beta)$, and $\|\widehat{C}_{XX}-C_{XX}\|_{\infty}$, $\|\widehat{C}_{YY}-C_{YY}\|_{\infty}$ and $\|\widehat{C}_{XY}-C_{XY}\|_{\infty}$ are all $O_p(T^{\beta})$ for some $\beta \in (0,1/2]$. Our proof of Theorem \ref{thmblp} corresponds to the particular case with $\beta=1/2$. 
						
						Note that $\sup_{j \geq 1}|\hat{\lambda}_j - \lambda_j| \leq \|\widehat{C}_{XX}-C_{XX}\|_{\infty} = O_p(T^{-\beta})$ (Lemma 4.2 of \citealp{Bosq2000}). Under Assumption \ref{assumpelbow}, we have $\hat{\lambda}_{k_T} - \hat{\lambda}_{k_T+1} \geq \tau_T$ and thus 
						\begin{equation}
							{T}^{\beta}(\hat{\lambda}_{k_T} - \hat{\lambda}_{k_T+1})  =  {T}^{\beta}(\hat{\lambda}_{k_T} - \lambda_{k_T} + \lambda_{k_T+1} - \hat{\lambda}_{k_T+1}) + {T}^{\beta}(\lambda_{k_T}-\lambda_{k_T+1}) \geq  T^{\beta}\tau_T. \label{pfeq001}
						\end{equation}
						If $\lambda_{k_T}= \lambda_{k_T+1}$, \eqref{pfeq001} reduces to ${T}^{\beta}(\hat{\lambda}_{k_T} - \hat{\lambda}_{k_T+1}) \geq T^{\beta}\tau_T$. Moreover, since ${T}^{\beta}(\hat{\lambda}_{k_T} - \hat{\lambda}_{k_T+1})=O_p(1)$ and $T^{\beta}\tau_T\to_p \infty$, $\mathbb{P}\{\hat{\lambda}_{k_T} - \hat{\lambda}_{k_T+1} \geq \tau_T | \lambda_{k_T}= \lambda_{k_T+1}\} \to 0$ as $T \to \infty$. Using the Bayes' rule and the facts that $\mathbb{P}\{\hat{\lambda}_{k_T} - \hat{\lambda}_{k_T+1} \geq \tau_T\} = 1$ and $\mathbb{P}\{\lambda_{k_T}= \lambda_{k_T+1}\}=\mathbb{P}\{\lambda_{k_T}= \lambda_{k_T+1}|\hat{\lambda}_{k_T} - \hat{\lambda}_{k_T+1} \geq \tau_T\}$ by Assumption \ref{assumpelbow}, we find that $\mathbb{P}\{ \lambda_{k_T}= \lambda_{k_T+1}\} \to 0$. 
						To establish the desired consistency, we thus may subsequently assume that \(\lambda_{k_T} > \lambda_{k_T+1}\). 
						
						Let $C_{\min} = C_{YY}- C_{YY}^{1/2}R_{XY}R_{XY}^\ast C_{YY}^{1/2}$. Since there is $A \in \mathcal L_{\infty}$ such that $C_{XY}=AC_{XX}$,  $C_{\min} = C_{YY}- AC_{XX}A^\ast$ (Corollary \ref{cor0}). From the triangular inequality applied to the $\SC$-norm, we find that
						\begin{align}
							\|\Sigma(\widehat{A})-C_{\min}\|_{\SC}  \notag 
							&\leq  {\|\widehat{C}_{YY}-C_{YY} \|_{\SC}} + \|\widehat{A}\widehat{C}_{XX,k_T}\widehat{A}^\ast  -A{C}_{XX,k_T} A^\ast\|_{\SC} \\ & \quad + \|{A}{C}_{XX,k_T}{A}^\ast  -A{C}_{XX} A^\ast\|_{\SC}.  \label{eqadd01}
						\end{align}
						It suffices to show that each summand in the RHS of \eqref{eqadd01} is \(o_p(1)\).
						
						We will first consider the first term in the RHS of \eqref{eqadd01}. Let $m_T$ be any divergent sequence (depending on $T$) but satisfy $T^{-\beta}m_T \to 0$. Note that 
						$\|\widehat{C}_{YY}\|_{\SC}-\|{C}_{YY}\|_{\SC} =  \sum_{j=1}^{{m}_T} (\hat{\mu}_j-{\mu}_j)  +  \sum_{j={m_T}+1}^{\infty} (\hat{\mu}_j-{\mu}_j) \leq m_T  \|\widehat{C}_{YY}-{C}_{YY}\|_{\infty} +  \sum_{j=m_T+1}^{\infty} (\hat{\mu}_j-{\mu}_j)$. For every $\delta > 0$, let $E_{\delta}=\{|\sum_{j=m_T+1}^{\infty} (\hat{\mu}_j-{\mu}_j)|  > \frac{\delta}{2}\}$ and $F_{\delta}=\{ m_T  \|\widehat{C}_{YY}-{C}_{YY}\|_{\infty} > \frac{\delta}{2}\}$. 
						Since $ \mathbb{P}\{ m_T  \|\widehat{C}_{YY}-{C}_{YY}\|_{\infty} +  \sum_{j=m_T+1}^{\infty} (\hat{\mu}_j-{\mu}_j)  > \delta \}\leq \mathbb{P}\{ E_{\delta}\}  + \mathbb{P}\{ F_{\delta}\}$, we find that 
						\begin{equation*}
							\mathbb{P}\{ \|\widehat{C}_{YY}\|_{\SC}-\|{C}_{YY}\|_{\SC} > \delta\}\leq
							\mathbb{P}\{ E_{\delta}\}  + \mathbb{P}\{ F_{\delta}\}.  
						\end{equation*}
						Note that $\widehat{C}_{YY}$ and $C_{YY}$ are Schatten 1-class operators (almost surely) and also $m_T$ increases without bound. Moreover, $m_T\|\widehat{C}_{YY}-C_{YY}\|_{\infty} = O_p(m_T T^{-\beta}) = o_p(1)$ under our assumptions. 
						These results imply that $\mathbb{P}\{E_{\delta}\} \to 0$ and  $\mathbb{P}\{F_{\delta}\} \to 0$ as $T \to \infty$, and thus we conclude that $\|\widehat{C}_{YY}\|_{\SC}-\|{C}_{YY}\|_{\SC} \to_p 0$. Combining this result with Lemma \ref{lema2} and the fact that $m_T\|\widehat{C}_{YY}-C_{YY}\|_{\infty} = o_p(1)$, we find that $\|\widehat{C}_{YY}-{C}_{YY}\|_{\SC} \to_p 0$ as desired. 
						
						We next consider the third term in the RHS of \eqref{eqadd01}. 
						We know from Lemma \ref{lema1} that \(\|{A}{C}_{XX,k_T}{A}^\ast  -A{C}_{XX} A^\ast\|_{\SC} \leq  \|A\|^2_{\infty} \|{C}_{XX,k_T} -{C}_{XX} \|_{\SC}\). Since  $C_{XX}$ is a Schatten 1-class operator and $k_T$ grows without bound,   $\|C_{XX,k_T} - C_{XX}\|_{\SC} = \sum_{j=k_T+1}^\infty \lambda_j \to_p 0$ and thus $ \|{A}{C}_{XX,k_T}{A}^\ast  -A{C}_{XX} A^\ast\|_{\SC} =o_p(1)$ as desired.

						We lastly  consider the second term in the RHS of \eqref{eqadd01}. Let $\varepsilon_t = Y_t - AX_t$.  Since $\widehat{C}_{XY} = A \widehat{C}_{XX} + \widehat{C}_{\varepsilon X}$ and $\widehat{A}\widehat{C}_{XX,k_T}\widehat{A}^\ast = 	\widehat{C}_{XY}\widehat{C}_{XX,k_T}^{-1}\widehat{C}_{XY}^\ast$,  we have 
						\begin{align*}
							\widehat{A}\widehat{C}_{XX,k_T}\widehat{A}^\ast  &= (A \widehat{C}_{XX} + \widehat{C}_{\varepsilon X}) \widehat{C}_{XX,k_T}^{-1} (A \widehat{C}_{XX} + \widehat{C}_{\varepsilon X})^\ast \notag \\ &= A \widehat{C}_{XX,k_T} A^\ast + A  \widehat{\Pi}_{X,k_T} \widehat{C}_{\varepsilon X}^\ast + \widehat{C}_{\varepsilon X} \widehat{\Pi}_{X,k_T} A^\ast + \widehat{C}_{\varepsilon X}\widehat{C}_{XX,k_T}^{-1} \widehat{C}_{\varepsilon X}^\ast,    
						\end{align*}
						where $\widehat{\Pi}_{X,k_T}  = \sum_{j=1}^{k_T} \hat{v}_j \otimes \hat{v}_j$. 
						Therefore, we have  
						\begin{align} 
							&\|\widehat{A}\widehat{C}_{XX,k_T}\widehat{A}^\ast  -A{C}_{XX,k_T} A^\ast\|_{\SC} \notag \\ &\leq \|{A}\widehat{C}_{XX,k_T}{A}^\ast  -A{C}_{XX,k_T} A^\ast\|_{\SC}  + \|A  \widehat{\Pi}_{X,k_T} \widehat{C}_{\varepsilon X}^\ast + \widehat{C}_{\varepsilon X} \widehat{\Pi}_{X,k_T} A^\ast + \widehat{C}_{\varepsilon X}\widehat{C}_{XX,k_T}^{-1} \widehat{C}_{\varepsilon X}^\ast\|_{\SC}. \label{eqadd2}\end{align}
						It will be proved later that the first term in the RHS of \eqref{eqadd2} is $o_p(1)$, i.e., 
						\begin{equation}
							\|{A}\widehat{C}_{XX,k_T}{A}^\ast  -A{C}_{XX,k_T} A^\ast\|_{\SC} = o_p(1). \label{eqadd2a}
						\end{equation} 
						We deduce from Lemma \ref{lema1}\eqref{lema1a} that the second term in the RHS of \eqref{eqadd2}, below denoted simply as $\mathcal {D}$, satisfies
						\begin{align} \label{eqadd3}
							\mathcal {D}\leq O(1)\| \widehat{\Pi}_{k_T}\widehat{C}_{\varepsilon X}^\ast \|_{\SC}  +  \|\widehat{C}_{\varepsilon X}\widehat{C}_{XX,k_T}^{-1} \widehat{C}_{\varepsilon X}^\ast\|^2_{\SC}.
						\end{align} 
						Since $\widehat{C}_{X\varepsilon} = T^{-1}\sum_{t=1}^T X_t \otimes Y_t - T^{-1} \sum_{t=1}^T X_t \otimes AX_t = C_{YX} - AC_{XX}  + O_p(T^{-\beta}) = O_p(T^{-\beta})$, we find that 
						\begin{align} \label{eqadd4}
							\max\left\{\|\widehat{\Pi}_{k_T}\widehat{C}_{\varepsilon X}^\ast\|_{\SC},  	\|\widehat{C}_{\varepsilon X}\widehat{\Pi}_{k_T}\|_{\SC}\right\} &\leq  \|\widehat{C}_{\varepsilon X}\|_{\infty} \|\widehat{\Pi}_{k_T}\|_{\SC} =  O_p(T^{-\beta}) k_{T} = o_p(1)  
						\end{align}
						and also
						\begin{align}
							\| \widehat{C}_{\varepsilon X}\widehat{C}_{XX,k_T}^{-1} \widehat{C}_{\varepsilon X}^\ast\|_{\SC}  &= \sum_{\ell=1}^\infty  \sum_{j=1}^{k_T} \hat{\lambda}_{j}^{-1} \langle \widehat{C}_{\varepsilon X} \hat{v}_j, \hat{v}_\ell \rangle^2 \leq \sum_{j=1}^{k_T} \hat{\lambda}_{j}^{-1} \| \widehat{C}_{\varepsilon X} \hat{v}_j\|^2  \leq \| \widehat{C}_{\varepsilon X}\|_{\infty}^2 \sum_{j=1}^{k_T} \hat{\lambda}_{j}^{-1} = o_p(1),\label{eqadd5}
						\end{align}
						where the last equality follows from the fact that $\| \widehat{C}_{\varepsilon X}\|_{\infty}^2 = O_p(T^{-2\beta})$ and $\sum_{j=1}^{k_T} \hat{\lambda}_j^{-1} \leq \tau_T^{-1} k_T = o(T^{2\beta})$ hold under Assumptions \ref{assumpelbow} and \ref{assum1}. As shown by \eqref{eqadd3}-\eqref{eqadd5},  the second term in the RHS of \eqref{eqadd2} is $o_p(1)$. Combining this result with  \eqref{eqadd2a}, we find that   $\|\widehat{A}\widehat{C}_{XX,k_T}\widehat{A}^\ast  -A{C}_{XX,k_T} A^\ast\|_{\SC} = o_p(1)$ as desired.
						
						It remains to verify \eqref{eqadd2a} to complete the proof. 
						From Lemma \ref{lema1}\eqref{lema1a}, we know that  \begin{equation}\label{eqpfimpadd}
							\|A \widehat{C}_{XX,k_T} A^\ast - A {C}_{XX,k_T} A^\ast \|_{\SC} \leq \|A\|_{\infty}^2 \|\widehat{C}_{XX,k_T}-{C}_{XX,k_T}\|_{\SC},
						\end{equation} 
						and thus it suffices to show that the RHS of \eqref{eqpfimpadd} is $o_p(1)$.  To this end, we first note that
						\begin{align}
							\|\widehat{C}_{XX,k_T}\|_{\SC}-\|{C}_{XX,k_T}\|_{\SC} =  \sum_{j=1}^{k_T} (\hat{\lambda}_j-{\lambda}_j)
							\leq k_T \|\widehat{C}_{XX}-{C}_{XX}\|_{\infty} \to_p 0. \label{eqpfimp}
						\end{align}
						We then obtain an upper bound of $\|\widehat{C}_{XX,k_T} - {C}_{XX,k_T}\|_{\infty}$ as follows:
						\begin{align}
							\|\widehat{C}_{XX,k_T} - {C}_{XX,k_T}\|_{\infty} &\leq \left\|\widehat{\Lambda}_{k_T} + \sum_{j=1}^{k_T} (\hat{\lambda}_j-\lambda_j) {f}_j \otimes {f}_j  \right\|_{\infty} \leq \left\|\widehat{\Lambda}_{k_T}\right\|_{\infty}  +  O_p(T^{-\beta}), \label{eqpfimp00}
						\end{align}
						where $\widehat{\Lambda}_{k_T} = \sum_{j=1}^{k_T} \hat{\lambda}_j \hat{f}_j \otimes \hat{f}_j - \sum_{j=1}^{k_T} \hat{\lambda}_j {f}_j \otimes {f}_j$.
						From similar algebra used in the proof of Lemma 3.1 of \cite{REIMHERR201562} and the fact that $\|\cdot\|_{\infty}\leq \|\cdot\|_{\SCC}$, we find that 
						\begin{equation} \label{eqpfimp01}
							\left\|\widehat{\Lambda}_{k_T}\right\|_{\infty}^2 \leq  \sum_{j=1}^{k_T} \hat{\lambda}_j^2 \sum_{\ell=k_T+1}^\infty  \langle \hat{f}_j, f_{\ell} \rangle^2 + \sum_{\ell=1}^{k_T} \hat{\lambda}_\ell^2 \sum_{j=k_T+1}^\infty  \langle \hat{f}_j, f_{\ell} \rangle^2 + \sum_{j=1}^{k_T} \sum_{\ell=1}^{k_T} (\hat{\lambda}_j - \hat{\lambda}_\ell)^2 \langle \hat{f}_j, f_{\ell} \rangle^2.
						\end{equation}
						Observe that, for every $\ell=1,\ldots,k_T$, 
						\begin{align} 
							(\hat{\lambda}_{\ell} - \hat{\lambda}_{\ell+1})^2 \sum_{j=k_T+1}^\infty   \langle \hat{f}_j, f_{\ell} \rangle^2&\leq \sum_{j=k_T+1}^\infty  (\hat{\lambda}_{\ell} - \hat{\lambda}_{j})^2  \langle \hat{f}_j, f_{\ell} \rangle^2 
							\notag \\ &= \sum_{j=k_T+1}^\infty  (\langle \hat{\lambda}_{\ell}\hat{f}_j,  f_{\ell} \rangle - \langle \widehat{C}_{XX} \hat{f}_j,  f_{\ell} \rangle)^2\notag\\
							&=\sum_{j=k_T+1}^\infty  (\langle \hat{f}_j,  (\hat{\lambda}_{\ell}-\lambda_{\ell})f_{\ell} \rangle + \langle \hat{f}_j,  (C_{XX} -  \widehat{C}_{XX})f_{\ell} \rangle)^2  \notag \\ 
							&\leq  \|(\hat{\lambda}_{\ell}-\lambda_{\ell})f_{\ell} + ({C}_{XX}-\widehat{C}_{XX})f_{\ell}\|^2. \label{eqpfimp01add} 
						\end{align}
						Since $\sup_{\ell\geq 1}|\hat{\lambda}_{\ell}-\lambda_{\ell}| \leq\|\widehat{C}_{XX}-C_{XX}\|_{\infty}=O_p(T^{-\beta})$, we find that the RHS of \eqref{eqpfimp01add} is $O_p(T^{-2\beta})$. Using the fact that  $\tau_T^{-1} \geq (\hat{\lambda}_\ell- \hat{\lambda}_{\ell+1})^{-1}$ for all $\ell=1,\ldots,k_T$, the following is deduced: for some $\delta > 0$, 
						\begin{equation}
							\sum_{\ell=1}^{k_T} {\hat{\lambda}_\ell^2} \sum_{j=k_T+1}^\infty  \langle \hat{f}_j, f_{\ell} \rangle^2 = O_p(T^{-2\beta}) \sum_{\ell=1}^{k_T} \frac{\hat{\lambda}_\ell^2 }{(\hat{\lambda}_{\ell}-\hat{\lambda}_{\ell+1})^2}  \leq O_p(\tau_T^{-2}T^{-2\beta})  \sum_{\ell=1}^{k_T} \hat{\lambda}_\ell^2  = O_p(T^{-\delta}), \label{eqpfimp02} 
						\end{equation}
						where the last equality is deduced from the facts that $\sum_{\ell=1}^{k_T} \hat{\lambda}_\ell^2 = O_p(1)$ and $\tau_T^{-1} = O_p(T^{\gamma_1})$ for some $\gamma_1 \in (0,\beta)$ under the employed conditions. From nearly identical arguments, we also find that 
						\begin{equation}
							\sum_{j=1}^{k_T} \hat{\lambda}_j^2 \sum_{\ell=k_T+1}^\infty  \langle \hat{f}_j, f_{\ell} \rangle^2 = O_p(T^{-\delta}). \label{eqpfimp03}
						\end{equation} Moreover, from similar algebra used in \eqref{eqpfimp01add} we find that 
						\begin{align}
							\sum_{j=1}^{k_T} \sum_{\ell=1}^{k_T} (\hat{\lambda}_j - \hat{\lambda}_\ell)^2 \langle \hat{f}_j, f_{\ell} \rangle^2 
							&=\sum_{j=1}^{k_T} \sum_{\ell=1}^{k_T} (\langle \hat{f}_j, \widehat{C}_{XX} f_{\ell} \rangle-\langle  \hat{f}_j, (\hat{\lambda}_\ell - \lambda_{\ell})f_{\ell} \rangle + \langle  \hat{f}_j, C_{XX}f_{\ell} \rangle)^2  
							\notag\\&\leq  \sum_{j=1}^{k_T}  \|(\hat{\lambda}_{\ell}-\lambda_{\ell})f_{\ell} + (\widehat{C}_{XX}-C_{XX})f_{\ell}\|^2
							= O_p(k_T T^{-2\beta}).  \label{eqpfimp04}
						\end{align}
						From \eqref{eqpfimp00}-\eqref{eqpfimp04}, we know that there is a divergent sequence $m_T$ such that $m_T \|\widehat{C}_{XX,k_T} - {C}_{XX,k_T}\|_{\infty} \to_p 0$. Combining this result with \eqref{eqpfimp} and Lemma \ref{lema2}, we find that the RHS of \eqref{eqpfimpadd} is $o_p(1)$ (and thus \eqref{eqadd2a} holds).
					\end{proof}		
						
						\begin{proof}[Proof of Corollary \ref{cor0}]
							(a) directly follows from the facts that (i) $\|(B C_{XX}^{1/2} - C_{YY}^{1/2} R_{XY})w_j\|^2 = 0$ for all $j \geq 1$ if and only if $ \| BC_{XX}^{1/2} -  C_{YY}^{1/2}R_{XY}\|_{\SCC}^2=0$ and (ii) the Hahn-Banach extension theorem (see e.g., Theorem 1.9.1 of \citealp{megginson1998}). This result  in turn implies (b) due to the fact that  \(C_{XY}  = {C}_{YY}^{1/2}R_{XY} {C}_{XX}^{1/2}\) as observed in Proposition \ref{prop1}. 
						\end{proof}
						
						\begin{proof}[Proof of Corollary \ref{cor2}]
							Note that \eqref{eqadd01} holds when $\widehat{A}$ is replaced by $\widetilde{A}$ , and $\|\widehat{C}_{YY}-C_{YY}\|_{\SC} = o_p(1)$  and $\|{A}{C}_{XX,k_T}{A}^\ast  -A{C}_{XX} A^\ast\|_{\SC} = o_p(1)$ can be shown as in our proof of Theorem \ref{thmblp}. 	
							We thus have \begin{equation*}
								\|\Sigma(\widetilde{A})-C_{\min}\|_{\SC} \leq  o_p(1) + \|\widetilde{A}\widehat{C}_{XX,k_T}\widetilde{A}^\ast  -A{C}_{XX,k_T} A^\ast\|_{\SC},  
							\end{equation*}
							and hence it suffices to show that $\|\widetilde{A}\widehat{C}_{XX,k_T}\widetilde{A}^\ast  -A{C}_{XX,k_T} A^\ast\|_{\SC} = o_p(1)$. Note that 
							\begin{equation*}
								\widetilde{A}\widehat{C}_{XX,k_T}\widetilde{A}^\ast = \widehat{\Pi}_{Y,\ell_T}(A \widehat{C}_{XX,k_T} A^\ast  + A  \widehat{\Pi}_{X,k_T} \widehat{C}_{\varepsilon X}^\ast  +  \widehat{C}_{\varepsilon X} A^\ast + \widehat{C}_{\varepsilon X}\widehat{C}_{XX,k_T}^{-1} \widehat{C}_{\varepsilon X}^\ast)\widehat{\Pi}_{Y,\ell_T}.    
							\end{equation*}
							Using similar arguments used in our proof of Theorem \ref{thmblp}, Lemma \ref{lema1} and the facts that $\|\widehat{\Pi}_{Y,\ell_T}\|_{\infty} \leq 1$ for any arbitrary $\ell_T\geq 1$ and $\|{A}{C}_{XX,k_T}{A}^\ast  -A{C}_{XX} A^\ast\|_{\SC} = o_p(1)$, we find the following: 
							(a) $\|\widehat{\Pi}_{Y,\ell_T}({A}\widehat{C}_{XX,k_T}{A}^\ast  -A{C}_{XX,k_T} A^\ast)\widehat{\Pi}_{Y,\ell_T}\|_{\SC} = o_p(1)$, 	(b) $\|\widehat{\Pi}_{Y,\ell_T}({A}{C}_{XX,k_T}{A}^\ast  -A{C}_{XX} A^\ast)\widehat{\Pi}_{Y,\ell_T}\|_{\SC} = o_p(1)$, and (c) $\|\widehat{\Pi}_{Y,\ell_T} (A  \widehat{\Pi}_{X,k_T} \widehat{C}_{\varepsilon X}^\ast  +  \widehat{C}_{\varepsilon X} A^\ast + \widehat{C}_{\varepsilon X}\widehat{C}_{XX,k_T}^{-1} \widehat{C}_{\varepsilon X}^\ast)\widehat{\Pi}_{Y,\ell_T}\|_{\SC} = o_p(1)$. Therefore, 
							\begin{equation*}
								\|	\widetilde{A}\widehat{C}_{XX,k_T}\widetilde{A}^\ast - \widehat{\Pi}_{Y,\ell_T}A{C}_{XX} A^\ast\widehat{\Pi}_{Y,\ell_T}\|_{\SC} \to_p 0.
							\end{equation*}
							If  $\|\widehat{\Pi}_{Y,\ell_T}A{C}_{XX} A^\ast\widehat{\Pi}_{Y,\ell_T} -A{C}_{XX} A^\ast \|_{\SC} \to_p 0$,  the desired result is established.  Let $\Upsilon  = C_{YY}^{1/2} R_{XY}$ and $\widehat{\Upsilon} = \widehat{\Pi}_{Y,\ell_T}  C_{YY}^{1/2} R_{XY}$. We then have $AC_{XX}A^\ast = \Upsilon \Upsilon^{\ast}$ (see Proposition \ref{prop1}) and $\widehat{\Pi}_{Y,\ell_T}A{C}_{XX} A^\ast\widehat{\Pi}_{Y,\ell_T}=\widehat{\Upsilon}\widehat{\Upsilon}^\ast$. Observe that 
							\begin{align}
								\|\widehat{\Upsilon}\widehat{\Upsilon}^\ast - {\Upsilon}{\Upsilon}^\ast\|_{\infty} = \|\widehat{\Upsilon} (\widehat{\Upsilon}^\ast-\Upsilon^\ast) + (\widehat{\Upsilon}-\Upsilon)\Upsilon^{\ast}\|_{\infty} &\leq O_p(1) \|\widehat{\Upsilon}-\Upsilon\|_{\infty} \notag \\ &\leq O_p(1) \| \widehat{\Pi}_{Y,\ell_T}  C_{YY}^{1/2}- C_{YY}^{1/2}\|_{\infty}.\notag
							\end{align}
							This implies that $\|\widehat{\Upsilon}\widehat{\Upsilon}^\ast - {\Upsilon}{\Upsilon}^\ast\|_{\infty}\to_p 0$ if there exists a divergent sequence $m_T$ such that $m_T\| \widehat{\Pi}_{Y,\ell_T}  C_{YY}^{1/2}- C_{YY}^{1/2}\|_{\infty} \to_p 0$ and $m_T \leq \ell_T$ for large $T$. Moreover, we find the following from the fact that \(\Upsilon\Upsilon^\ast\) is nonnegative self-adjoint and \(\{\hat{w}_j\}_{j=1}^\infty\) is an orthonormal basis of \(\mathcal H\): 
							\begin{equation*}
								\|\Upsilon\Upsilon^\ast\|_{\SC} - \|\widehat{\Upsilon}\widehat{\Upsilon}\|_{\SC} = \sum_{j=\ell_T+1}^{\infty} \langle \Upsilon\Upsilon^\ast \hat{w}_j, \hat{w}_j \rangle \leq \sum_{j=m_T+1}^{\infty} \langle \Upsilon\Upsilon^\ast \hat{w}_j, \hat{w}_j \rangle,
							\end{equation*} 
							which is $o_p(1)$ since $\Upsilon\Upsilon^\ast = C_{YY}^{1/2} R_{XY}R_{XY}^\ast C_{YY}^{1/2}$ is a Schatten 1-class (Lemma \ref{lema1}\eqref{lema1a}). We thus deduce from Lemma \ref{lema2} that $\|\widehat{\Upsilon}\widehat{\Upsilon}^\ast - {\Upsilon}{\Upsilon}^\ast\|_{\SC} \to_p 0$ as desired. 
						\end{proof}

						\section{Additional simulation results for alternative estimators} \label{sec_addmonte}
						In this section, we experiment with an alternative estimator \(\widetilde{A}\) which is introduced in Section \ref{sec_reduction}. 
						We replicate the same simulation experiments conducted in Section \ref{secest2} using the alternative estimator \(\widetilde{A}\) introduced in Section \ref{sec_reduction}.  We construct \(\widehat{\Pi}_{Y,\ell_T}\) using the eigenvectors \(\{\hat{u}_j\}_{j\geq 1}\) of \(\widehat{C}_{YY}\) with \(\ell_T = \lfloor T^{1/2} \rfloor\). Overall, we found that the computed excess MSPEs tend to be similar to those obtained with \(\widehat{A}\), providing supporting evidence for our finding in Corollary \ref{cor2}. The simulation results are reported in Table \ref{tab2}.  We also experimented with \(\widehat{\Pi}_{Y,\ell_T}\) constructed from the eigenvectors \(\{\hat{v}_j\}_{j\geq 1}\) of \(\widehat{C}_{XX}\), and the results are reported in Table \ref{tab3}.  Even if the empirical MSPEs tend to be larger than in the previous case, we obtained overall similar simulation results.  
						

						\begin{table}[H]
							\small
							\caption{Excess MSPE of the alternative predictor, $\widehat{\Pi}_{Y,\ell_T} = \sum_{j=1}^{\ell_T} \hat{u}_j\otimes\hat{u}_j$}  
							\label{tab2}
							\renewcommand*{\arraystretch}{1} \vspace{-0.1em}
							\begin{minipage}{\linewidth}
								\centering
								\subcaptionbox{Case BB\vspace{-0.55em}}{%
									\vspace{0em} 
									\begin{tabular*}{1\linewidth}{@{\extracolsep{\fill}}c@{\hspace{1\tabcolsep}}| ccccc@{\hspace{1\tabcolsep}}|ccccc}
										\hline\hline
										& &\multicolumn{3}{c}{$\gamma= 0.475$}& & &\multicolumn{3}{c}{$\gamma= 0.45$}&\\ \hline 
										\hspace{-0.25cm}	$T$ & $50$ & $100$ & $200$ &{$400$} & {$800$} & $50$ & $100$ & $200$ &{$400$} & {$800$} \\
										\hline 
										\hspace{-0.1cm}	Model 1& 0.043& 0.035& 0.023& 0.017 &0.011& 0.071& 0.049& 0.029& 0.021& 0.013 \\
										\hspace{-0.1cm}	Model 2 & 0.045& 0.035& 0.023& 0.017& 0.012& 0.071& 0.051& 0.030 &0.021& 0.014 \\
										\hspace{-0.1cm}	Model 3 & 0.056 &0.045& 0.030& 0.022& 0.015& 0.091 &0.063 &0.038 &0.027& 0.017 \\
										\hspace{-0.1cm}	Model 4 &0.057 &0.046 &0.031 &0.023& 0.015& 0.092 &0.065& 0.039& 0.028& 0.018 \\
										\hline\hline
									\end{tabular*}
								}
							\end{minipage}\\[0.4em]
							
							\begin{minipage}{\linewidth}
								\centering
								\subcaptionbox{Case CBM\vspace{-0.55em}}{%
									\vspace{0em} 
									\begin{tabular*}{1\linewidth}{@{\extracolsep{\fill}}c@{\hspace{1\tabcolsep}}| ccccc@{\hspace{1\tabcolsep}}|ccccc}
										\hline\hline
										& &\multicolumn{3}{c}{$\gamma= 0.475$}& & &\multicolumn{3}{c}{$\gamma= 0.45$}&\\ \hline 
										\hspace{-0.1cm}	$T$ & $50$ & $100$ & $200$ &{$400$} & {$800$} & $50$ & $100$ & $200$ &{$400$} & {$800$} \\
										\hline 
										\hspace{-0.1cm}	Model 1& 0.043& 0.035& 0.023& 0.017& 0.011&  0.071& 0.050 &0.029& 0.021& 0.013 \\
										\hspace{-0.1cm}	Model 2 &0.044& 0.036& 0.023 &0.017 &0.012&  0.072 &0.051& 0.030 &0.021& 0.014 \\
										\hspace{-0.1cm}	Model 3 & 0.057 &0.046& 0.030& 0.022 &0.015 &0.093& 0.065& 0.039 &0.027& 0.018 \\
										\hspace{-0.1cm}	Model 4 &0.057& 0.047 &0.031& 0.023& 0.015& 0.094& 0.066& 0.039 &0.028& 0.018 \\
										\hline\hline
									\end{tabular*}
								}
							\end{minipage}\\[0.4em]
							
							\begin{minipage}{\linewidth}
								\centering
								\subcaptionbox{Case BM\vspace{-0.55em}}{%
									\vspace{0em} 
									\begin{tabular*}{1\linewidth}{@{\extracolsep{\fill}}c@{\hspace{1\tabcolsep}}| ccccc@{\hspace{1\tabcolsep}}|ccccc}
										\hline\hline
										& &\multicolumn{3}{c}{$\gamma= 0.475$}& & &\multicolumn{3}{c}{$\gamma= 0.45$}&\\ \hline 
										$T$ & $50$ & $100$ & $200$ &{$400$} & {$800$} & $50$ & $100$ & $200$ &{$400$} & {$800$} \\
										\hline 
										\hspace{-0.1cm}	Model 1&  0.012& 0.010& 0.006 &0.004& 0.004& 0.026& 0.017& 0.009& 0.006& 0.004 \\
										\hspace{-0.1cm}	Model 2&0.012& 0.010& 0.006& 0.004 &0.004 & 0.027 &0.018& 0.009& 0.006& 0.004 \\
										\hspace{-0.1cm}	Model 3 & 0.018 &0.014 &0.009& 0.006& 0.005 &0.037 &0.024 &0.013 &0.008& 0.006 \\
										\hspace{-0.1cm}	Model 4 &0.018& 0.015& 0.009& 0.006& 0.005 & 0.038& 0.025 &0.013 &0.008& 0.006 \\
										\hline\hline
									\end{tabular*}
								}
							\end{minipage}
							
							\vspace{0.4em}
							\begin{minipage}{1\linewidth}
								{\footnotesize Notes: The excess MSPEs are calculated as in Table \ref{tab1}.} 
							\end{minipage}
						\end{table}

						\begin{table}[H]
							\small
							\caption{Excess MSPE of the alternative predictor, $\widehat{\Pi}_{Y,\ell_T} = \sum_{j=1}^{\ell_T} \hat{v}_j\otimes\hat{v}_j$}  
							\label{tab3}
							\renewcommand*{\arraystretch}{1} \vspace{-0.1em}
							\begin{minipage}{\linewidth}
								\centering
								\subcaptionbox{Case BB\vspace{-0.55em}}{%
									\vspace{0em} 
									\begin{tabular*}{1\linewidth}{@{\extracolsep{\fill}}c@{\hspace{1\tabcolsep}}| ccccc@{\hspace{1\tabcolsep}}|ccccc}
										\hline\hline
										& &\multicolumn{3}{c}{$\gamma= 0.475$}& & &\multicolumn{3}{c}{$\gamma= 0.45$}&\\ \hline 
										$T$ & $50$ & $100$ & $200$ &{$400$} & {$800$} & $50$ & $100$ & $200$ &{$400$} & {$800$} \\
										\hline 
										\hspace{-0.1cm}	Model 1& 0.043& 0.035& 0.023& 0.017& 0.011& 0.071& 0.049 &0.029& 0.021& 0.013 \\
										\hspace{-0.1cm}	Model 2 &0.066& 0.049& 0.032& 0.023 &0.015& 0.089 &0.063& 0.039& 0.026& 0.017 \\
										\hspace{-0.1cm}	Model 3 &0.056& 0.045& 0.030& 0.022 &0.015& 0.091& 0.063 &0.038& 0.027& 0.017 \\
										\hspace{-0.1cm}	Model 4 &0.078 &0.057& 0.036& 0.025& 0.017 & 0.109& 0.075 &0.044 &0.030& 0.019 \\
										\hline\hline
									\end{tabular*}
								}
							\end{minipage}\\[0.4em]
							
							\begin{minipage}{\linewidth}
								\centering
								\subcaptionbox{Case CBM\vspace{-0.55em}}{%
									\vspace{0em} 
									\begin{tabular*}{1\linewidth}{@{\extracolsep{\fill}}c@{\hspace{1\tabcolsep}}| ccccc@{\hspace{1\tabcolsep}}|ccccc}
										\hline\hline
										& &\multicolumn{3}{c}{$\gamma= 0.475$}& & &\multicolumn{3}{c}{$\gamma= 0.45$}&\\ \hline 
										$T$ & $50$ & $100$ & $200$ &{$400$} & {$800$} & $50$ & $100$ & $200$ &{$400$} & {$800$} \\
										\hline 
										\hspace{-0.1cm}	Model 1& 0.044& 0.035& 0.023& 0.017 &0.011& 0.071& 0.050& 0.029& 0.021& 0.013 \\
										\hspace{-0.1cm}	Model 2 &0.075& 0.055 &0.036 &0.025& 0.016& 0.099 &0.069& 0.043& 0.029& 0.018 \\
										\hspace{-0.1cm} Model 3 &0.057& 0.046& 0.030& 0.022& 0.015& 0.093& 0.065& 0.039& 0.027& 0.018 \\
										\hspace{-0.1cm}	Model 4 &0.096& 0.071 &0.046& 0.032 &0.021& 0.127& 0.089 &0.055 &0.037& 0.023 \\
										\hline\hline
									\end{tabular*}
								}
							\end{minipage}\\[0.4em]
							
							\begin{minipage}{\linewidth}
								\centering
								\subcaptionbox{Case BM\vspace{-0.55em}}{%
									\vspace{0em} 
									\begin{tabular*}{1\linewidth}{@{\extracolsep{\fill}}c@{\hspace{1\tabcolsep}}| ccccc@{\hspace{1\tabcolsep}}|ccccc}
										\hline\hline
										& &\multicolumn{3}{c}{$\gamma= 0.475$}& & &\multicolumn{3}{c}{$\gamma= 0.45$}&\\ \hline 
										$T$ & $50$ & $100$ & $200$ &{$400$} & {$800$} & $50$ & $100$ & $200$ &{$400$} & {$800$} \\
										\hline 
										\hspace{-0.1cm}	Model 1& 0.012 &0.010 &0.006& 0.004& 0.004& 0.026 &0.017& 0.009 &0.006& 0.004 \\
										\hspace{-0.1cm}	Model 2 &0.032& 0.022& 0.014& 0.009& 0.007& 0.046& 0.030& 0.017 &0.011& 0.008 \\
										\hspace{-0.1cm}	Model 3&0.018& 0.015& 0.009& 0.006& 0.005& 0.037& 0.024& 0.013& 0.008& 0.006 \\
										\hspace{-0.1cm}	Model 4 &0.040& 0.028& 0.017& 0.011& 0.008& 0.058& 0.037& 0.021& 0.013& 0.009 \\
										\hline\hline
									\end{tabular*}
								}
							\end{minipage}
							
							\vspace{0.4em}
							\begin{minipage}{1\linewidth}
								{\footnotesize Notes: The excess MSPEs are calculated as in Table \ref{tab1}.} 
							\end{minipage}
						\end{table}

						\bibliographystyle{apalike}
						\bibliography{../../../../../bibtexlib/swkrefs}

					\end{document}